\documentclass{article}
\usepackage{amsmath,amssymb}
\usepackage{amsthm}
\usepackage{enumerate}
\usepackage{url}
\usepackage{color}
\usepackage{mathtools}
\usepackage{verbatim}
\usepackage{tikz}
\usepackage{tabularray}

\usepackage[margin=1in]{geometry}

\newtheorem{theorem}{Theorem}[section]

\newtheorem{lemma}[theorem]{Lemma}
\newtheorem{corollary}[theorem]{Corollary}
\newtheorem{proposition}[theorem]{Proposition}

\theoremstyle{definition}
\newtheorem{definition}[theorem]{Definition}
\newtheorem{example}[theorem]{Example}
\newtheorem{notation}[theorem]{Notation}
\newtheorem{algorithm}[theorem]{Algorithm}
\newtheorem{remark}[theorem]{Remark}
\newtheorem*{remark*}{Remark}
\newtheorem{question}[theorem]{Question}

\newcommand{\abar}{\bar{a}}
\newcommand{\bbar}{\bar{b}}

\newcommand{\xbar}{\bar{x}}
\newcommand{\ybar}{\bar{y}}

\newcommand{\nv}{\text{-}}
\newcommand{\p}[2]{\mathfrak{t}({#1},{#2})}

\def\seq{\subseteq}

\def\N{\mathbb{N}}
\def\Z{\mathbb{Z}}
\def\Q{\mathbb{Q}}
\def\R{\mathbb{R}}
\def\T{\mathbb{T}}
\def\cC{\mathcal{C}}
\def\cL{\mathcal{L}}

\def\cB{\mathcal{B}}
\def\cP{\mathcal{P}}
\def\cD{\mathcal{D}}
\def\cF{\mathcal{F}}

\def\cM{\mathcal{M}}
\def\cN{\mathcal{N}}
\def\cS{\mathcal{S}}
\def\cR{\mathcal{R}}

\def\cZ{\mathcal{Z}}

\renewcommand{\phi}{\varphi}
\renewcommand{\H}{\mathbb{H}}
\newcommand{\inv}{^{\text{-}1}}
\def\VC{\operatorname{VC}}

\def\br{\operatorname{br}}
\def\im{\operatorname{im}}
\newcommand{\clqed}{\hfill$\dashv_{\text{\scriptsize{claim}}}$}

\usepackage{xcolor}

\title{VC-dimension of  generalized progressions in some nonabelian groups}
\author{Gabriel Conant\\
\small{University of Illinois Chicago}
\and 
Ay\c{c}in Iplik\c{c}i Arodirik\\
\small{The Ohio State University}
\and 
Tora Ozawa\\
\small{University of Rochester}
\and 
David Zeng\\
\small{University of Cambridge}}

\date{May 28, 2025}

\begin{document}

\maketitle
\abstract{We analyze generalized progressions in some nonabelian groups using a measure of complexity called VC-dimension, which was originally introduced in statistical learning theory by Vapnik and Chervonenkis. Here by a ``generalized progression" in a group $G$, we mean a finite subset of $G$ built from a fixed set of generators in analogy to a (multidimensional) arithmetic progression of integers. These sets play an important role in additive combinatorics and, in particular, the study of approximate groups. Our two main results establish finite upper bounds on the VC-dimension of certain set systems of generalized  progressions in finitely generated free groups and also the Heisenberg group over $\mathbb{Z}$.}

\section{Introduction}

In 1968, Vapnik and Chervonenkis \cite{VCdimRus} introduced a measure of complexity for classes of sets (or ``set systems"), which is now called \emph{VC-dimension}. Roughly speaking, the VC-dimension of a set system $\cS$ is the largest size (if it exists) of a finite set whose full powerset can be obtained by intersecting with the elements of $\cS$. (See Section \ref{sec:set-systems} for complete definitions.) The original work of Vapnik and Chervonenkis was based in the theory of empirical processes, and Blumer, Ehrenfeucht, Haussler, and Warmuth \cite{BEHW} later established a fundamental connection to machine learning (namely,  Valiant's notion of  PAC learning \cite{Valiant}). Set systems of bounded VC-dimension also play an important role in various areas of combinatorics such as property testing and graph regularity \cite{ADPR, AFKS, AFN, GGR, LovSzeg}, as well as related areas of discrete geometry such as discrepancy theory \cite{MWW}. In model theory, VC-dimension is directly tied to Shelah's notion of ``NIP theories" (as demonstrated by Laskowski \cite{LaskNIP}).

A canonical example is the set system of axis-parallel boxes in $\R^k$, which has VC-dimension $2k$ \cite{WenDu} (see Example \ref{ex:box-R}). Other examples include the set system of half-spaces in $\R^k$, which has VC-dimension $k+1$ (this a standard exercise involving Radon's Theorem), and the set system of regular convex $n$-gons in $\R^2$, which has VC-dimension $2n+1$ (see \cite[Example 3.2.2]{BEHW}). Further examples of a similar flavor are discussed in Section \ref{sec:examples}. Examples such as these illustrate the general expectation that set systems with a high level of structure (e.g., geometrically tame objects in Euclidean space) should have bounded VC-dimension. The purpose of this paper is to explore this expectation in the setting of discrete groups. In particular, the basic objects we are interested in are ``generalized progressions", which we now define.

\begin{definition}\label{def:GAP}
Let $G$ be a group. 
\begin{enumerate}[$(1)$]
\item Given $a_1,\ldots,a_k\in G$ and $N_1,\ldots,N_k\in\N$, let $P(\abar,\bar{N})$ denote the set of elements of $G$ that can be represented by a word in $a_1,\ldots,a_k,a_1\inv,\ldots,a_k\inv$ in which $a_i$ and $a_i\inv$  together appear at most $N_i$ times for all $1\leq i\leq k$. 
\item A \textbf{generalized progression} in $G$ is a set of the form $P(\abar,\bar{N})$ for some $a_1,\ldots,a_k\in G$ and $N_1,\ldots,N_k\in \N$. In this case, $k$ is  the \textbf{rank} of $P(\abar,\bar{N})$, and $a_1,\ldots,a_k$ are  \textbf{generators} of $P(\abar,\bar{N})$.

\item Given $a_1,\ldots,a_k\in G$, define $\cP_G(\abar)$ to be the collection of all left translates of all generalized progressions in $G$ generated by $a_1,\ldots,a_k$. 
 \end{enumerate}
\end{definition}

For example, if $a_1,\ldots,a_k$ are the standard basis generators of $\Z^k$, then $\cP_{\Z^k}(\abar)$ is a discrete version of axis-parallel boxes in $\R^k$. Moreover, one can show that the VC-dimension of $\cP_{\Z^k}(\abar)$ is also $2k$ using essentially the same proof. A related result for  certain generalized progressions in abelian groups is given by Sisask in \cite{Sisask} (see Remark \ref{rem:abelian}). 

In this paper, we are interested in proving finiteness of the VC-dimension of $\cP_G(\abar)$ in the case that $G$ is nonabelian. This interest is  motivated both by the apparent dearth of such examples in the existing literature, and also by the important role that generalized progressions play in certain ``arithmetic regularity" results on VC-dimension in arbitrary groups (see Section \ref{sec:motivation} for further details).  We will consider two extremes in the spectrum of   nonabelian groups, namely, finitely generated free groups and the Heisenberg group over $\Z$ (i.e., the free $2$-generated $2$-nilpotent group). The following are our two main results.

\begin{theorem}\label{thm:main2}
Let $\mathbb{H}$ denote the Heisenberg group over $\Z$ with generators $A=\left(\begin{smallmatrix}
1 & 1 & 0\\
0 & 1 & 0\\
0 & 0 & 1
\end{smallmatrix}\right)$ and $B=\left(\begin{smallmatrix}
1 & 0 & 0\\
0 & 1 & 1\\
0 & 0 & 1
\end{smallmatrix}\right)$. Then the VC-dimension of $\cP_{\mathbb{H}}(A,B)$ is at most $267$.
\end{theorem}

\begin{theorem}\label{thm:main1}
Let $F_k$ denote the free group on $k$ generators $a_1,\ldots,a_k$. Then the VC-dimension of $\cP_{F_k}(\bar{a})$ is at most $3k-1$.
\end{theorem}

The proof strategies for these two results are quite different. In the case of the Heisenberg group, we will first give an explicit arithmetic description of generalized progressions generated by $A$ and $B$ (see Theorem \ref{thm:PNNexp}), which shows that $\cP_{\mathbb{H}}(A,B)$ is uniformly quantifier-free definable in a certain first-order structure whose quantifier-free formulas are NIP (a model-theoretic manifestation of VC-dimension; see Section \ref{sec:MT}). From this, finiteness of the VC-dimension of $\cP_{\mathbb{H}}(A,B)$ follows immediately, but with no explicit bound (see Corollary \ref{cor:VCP}). To obtain such a bound, we will employ effective results of Karpinski and Macintyre \cite{KarMac} on the VC-dimension of semialgebraic families in $\R^n$.

For the case of free groups, we will use perspectives from geometric group theory  to formulate constraints on VC-dimension in terms of combinatorial properties of the underlying Cayley graph (see Proposition \ref{prop:structure_of_3k_weak}). Thus, unlike the case of the Heisenberg group, our analysis of the free group will not require model-theoretic notions. That being said, the model theory of free groups is a rich area of study involving highly sophisticated machinery from geometric group theory, and so it is natural to wonder if our investigation of VC-dimension fits into the model-theoretic picture in a way analogous to the Heisenberg group. We will elaborate on this possibility in Remark \ref{rem:free-group-NIP}.

The paper is structured as follows. In Sections \ref{sec:prelim1} and \ref{sec:prelim2} we establish some basic preliminaries on VC-dimension, both in general and then in the specialized setting of groups. Theorems \ref{thm:main2} and \ref{thm:main1} are then proved in Sections \ref{sec:Heisenberg} and \ref{sec:free} (respectively).

\subsection*{Acknowledgements}
The results of this paper were largely obtained during an undergraduate research program (ROMUS) at the The Ohio State University in the summer of 2023. We thank OSU for a hospitable working environment, and the Department of Mathematics for significant financial support. Support was also provided by the NSF through the first author's grant (DMS-1855503).

\section{Preliminaries on VC-dimension}\label{sec:prelim1}
 
 \subsection{Notation}

 Given a set $X$ and an integer $n\geq 0$, let ${X\choose n}$ denote the set of subsets of $X$ of size $n$. Given an integer $n\geq 1$, let $[n]=\{1,\ldots,n\}$.

 \subsection{General set systems}\label{sec:set-systems}
 
A \textbf{set system} on a set $X$ is a collection $\cS\seq\cP(X)$ of subsets of $X$. 

 \begin{definition}
 Let $\cS$ be a set system on a  set $X$.
 \begin{enumerate}[$(1)$]
     \item Given $A\seq X$ and $B\seq A$, we say $\cS$ \textbf{cuts out $B$ from $A$} if there is some $S\in\cS$ such that $S\cap A=B$.
     \item Given $A\seq X$, we say $\cS$ \textbf{shatters} $A$ if $\cS$ cuts out every subset of $A$ from $A$.
     \item The \textbf{VC-dimension} of $\cS$ is $\VC(\cS)=\sup\left\{n\in\N:\textnormal{$\cS$ shatters some set in ${X\choose n}$}\right\}$.
 \end{enumerate}
 \end{definition}

 We allow $\VC(\cS)$ to take the value $\infty$ when $\cS$ shatters subsets of $X$ of arbitrarily large size.

 \begin{definition}
 Let $\cS$ be a set system on a  set $X$. The \textbf{shatter function of $\cS$} is the function $\pi_{\cS}\colon \N\to \N$ so that $\pi_{\cS}(n)=\max\left\{|\{S\cap A:S\in\cS\}|:A\in {X\choose n}\right\}$.
 \end{definition}
 
 Note that in the context of the previous definition, if $n\geq 0$ then $\pi_{\cS}(n)\leq 2^n$ and, moreover, $\pi_{\cS}(n)=2^n$ if and only if $\cS$ shatters some some set in ${X\choose n}$. Therefore:
 \[
 \textstyle \VC(\cS)=\sup\{n\in\N:\pi_{\cS}(n)=2^n\}.
 \] 
The following is a good exercise in working with these definitions (or see \cite[Proposition 2.2]{Assouad}).

 \begin{proposition}\label{prop:pullback}
     Let $f\colon X\to Y$ be a  function. Fix a set system $\cS$ on $Y$ and let $f\inv(\cS)=\{f\inv (S):S\in\cS\}$. Then $\pi_{f\inv(\cS)}(n)\leq  \max_{0\leq k\leq n}\pi_{\cS}(k)$ for any $n\geq 0$, and thus $\VC(f\inv(\cS))\leq\VC(\cS)$. If $f$ is also surjective, then $ \pi_{\cS}(n)\leq \pi_{f\inv(\cS)}(n)$ for all $n\geq 0$, and thus  $\VC(\cS)=\VC(f\inv(\cS))$. 
 \end{proposition}

 In general, if $\VC(\cS)=d<\infty$ then $\pi_{\cS}(n)<2^n$ for all $n>d$. The Sauer-Shelah lemma says that, in fact, $\pi_{\cS}(n)$ can be bounded by a polynomial of degree $d$ in this case. Before stating this result, we define some notation that will be used throughout the paper.
 
 \begin{definition}
 Given an integer $d\geq 0$, define $\mathfrak{C}_d(n)=\sum_{i=0}^d{n\choose i}$. 
 \end{definition}
 
  \begin{remark}\label{rem:taylor}
 By the binomial theorem,  $\mathfrak{C}_d(n)\leq (n+1)^d$. 
A  better bound is $\mathfrak{C}_d(n)\leq (en/d)^d$, which holds for any $n\geq d$. This can be derived from the Maclaurin series for $e^x$. See also \cite[Lemma 4.3]{Vidyasagar}.
 \end{remark}
 
 \begin{lemma}[Sauer-Shelah]\label{lem:SSL}
Let $\cS$ be a set system on a set $X$. If $\VC(\cS)=d<\infty$,  then $\pi_{\cS}(n)\leq\mathfrak{C}_d(n)$ for all $n\geq 0$.
 \end{lemma} 
 
As is well known (e.g., \cite[Theorem 9.2.6]{DudNIP}),  the Sauer-Shelah lemma can be used to show that set systems of finite VC-dimension are closed under Boolean combinations in the following sense.

 \begin{definition}
 Given set systems $\cS_1$ and $\cS_2$ on a set $X$, define $\cS_1\wedge\cS_2=\{S_1\cap S_2:S_1\in\cS_1,~S_2\in\cS_2\}$.
 \end{definition}
 
 The next result (which is stated implicitly in the proof of \cite[Theorem 9.2.6]{DudNIP}) is a straightforward exercise; see also the proof of \cite[Lemma 3.2.3]{BEHW} or \cite[Theorem 4.5]{Vidyasagar}.
 
 \begin{proposition}\label{prop:intersection}
 If $\cS_1,\cS_2\seq \cP(X)$ then $\pi_{\cS_1\wedge\cS_2}(n)\leq\pi_{\cS_1}(n)\pi_{\cS_2}(n)$ for all $n\geq 0$.
 \end{proposition}
 
 \begin{remark}\label{rem:VCcomp}
Given $\cS\seq\cP(X)$, define $\neg\cS=\{X\backslash S:S\in\cS\}$. Then we have $\pi_{\cS}(n)=\pi_{\neg\cS}(n)$ for all $n\geq 0$. Thus the previous proposition holds for ``union families" as well by De Morgan's laws.
\end{remark}

\begin{definition}
Given $d,k\geq 0$, define $\mathfrak{f}(d,k)=\min\{n\in\N:k\log(\mathfrak{C}_d(n))< n\}$.
\end{definition}

\begin{corollary}\label{cor:VCintersection}
Suppose $\cS_1,\ldots,\cS_k\seq\cP(X)$ each have VC-dimension at most $d$. Then 
\[
\VC(\cS_1\wedge\ldots\wedge\cS_k)<\mathfrak{f}(d,k).
\]
\end{corollary}
\begin{proof}
Let $\cS=\cS_1\wedge\ldots\wedge\cS_k$.
By Lemma \ref{lem:SSL} and Proposition \ref{prop:intersection}, we have $\pi_{\cS}(n)\leq \mathfrak{C}_d(n)^k$ for all $n\geq 0$. In particular, $\pi_S(\mathfrak{f}(d,k))\leq (\mathfrak{C}_d(\mathfrak{f}(d,k)))^k<2^{\mathfrak{f}(d,k)}$, which implies $\VC(\cS)<\mathfrak{f}(d,k)$.
\end{proof}

\begin{remark}\label{rem:fbounds}
Via elementary arguments, one can establish the bounds 
\[
dk\log k< \mathfrak{f}(d,k)\leq dk\log(ck\log(ck))
\]
where $c=2^{\frac{1}{dk}}(e+\log(e))$ (the upper bound uses Remark \ref{rem:taylor}). 
\end{remark}

\subsection{Examples}\label{sec:examples}

\begin{example}[axis-parallel boxes in $\R^k$]\label{ex:box-R}
Let $\cS_b(\R^k)$ denote the set system on $\R^k$ consisting of all sets of the form $I_1\times\ldots\times I_k$, where each $I_t$ is an interval in $\R$. Then $\VC(\cS_b(\cR^k))=2k$. This was originally shown by Wenocur and Dudley \cite{WenDu} (we caution the reader that this source uses the notation $V(\cS)$ for $\VC(\cS)+1$). 
\end{example}

\begin{example}[axis-parallel cubes in $\R^k$]\label{ex:cube-R}
Let $\cS_c(\R^k)$ denote the set system on $\R^k$ consisting of all sets of the form $I_1\times\ldots\times I_k$, where each $I_t$ is a bounded interval in $\R$ of the same length. Then $\VC(\cS_c(\R^k))=\lfloor (3k+1)/2\rfloor$. This result first appears in a preprint by Despres \cite{Despres}. However, a flaw in the proof of the lower bound was later identified in the preprint \cite{GGK}, which also announces a correct proof.
\end{example}

\begin{example}[axis-parallel boxes/cubes in $\T^k$]\label{ex:box-torus}
 Let $\T^k$ denote the $k$-dimensional torus (i.e., the Cartesian product of $k$ copies of the unit circle $S^1$). Define $\cS_b(\T^k)$ to be the set system on $\R^k$ consisting of all sets of the form $I_1\times\ldots\times I_k$, where each $I_t$ is a cyclic interval in $S^1$. Define $\cS_c(\T^k)$ analogously, but while also assuming each $I_t$ has the same arclength. An exact formula for the VC-dimension of these set systems is unknown. However, Gillibert, Lachmann, and M\"{u}llner \cite{GLM} prove the following bounds:
 \[
 k\log k-4k\log\log k\leq \VC(\cS_c(\T^k))\leq \VC(\cS_b(\T^k))\leq k\log k+3k\log\log k.
 \]
\end{example}

\subsection{VC-dimension in model theory}\label{sec:MT}

Although we will not need any complicated results from model theory, it will be convenient in Section \ref{sec:Heisenberg} to use some basic model-theoretic concepts. Thus in this section we will assume knowledge of fundamental notions from first-order logic, including languages, structures, and substructures. For the reader not familiar with these notions, the material in Chapter 1 of \cite{Marker} is more than sufficient.

Let $\cL$ be a first-order language. We will work with \emph{bi-parititioned} $\cL$-formulas, by which we mean an $\cL$-formula $\varphi(\xbar,\ybar)$ whose free variables have been partitioned into two disjoint tuples $\xbar$ and $\ybar$. Given a set $X$ and a (finite) tuple $\xbar$ of variables we let $X^{\xbar}$ denote $X^{|\xbar|}$. 

\begin{definition}
    Let $\cM$ be an $\cL$-structure with universe $M$, and suppose $\varphi(\xbar,\ybar)$ is a bipartitioned $\cL$-formula (possibly with parameters from $M$). Given $\bbar\in M^{\ybar}$, let $\varphi(M,\bbar)=\{\abar\in M^{\xbar}:\cM\models\varphi(\abar,\bbar)$.  Define the set system
    \[
    \cS^{\cM}_{\varphi(\xbar,\ybar)}=\{\varphi(M,\bbar):\bbar\in M^{\ybar}\}.
    \]    
    We also write $\cS^{\cM}_\varphi$ when there is no possibility for confusion. Note that $\cS^{\cM}_\varphi$ is a set system on  $M^{\xbar}$. Finally, we let $\pi^{\cM}_{\varphi}$ denote $\pi_{\cS^{\cM}_{\varphi}}$.
\end{definition}

\begin{proposition}\label{prop:substructure}
    Suppose $\cM$ and $\cN$ are $\cL$-structures, with $\cM$ a substructure of $\cN$. Let $\varphi(\xbar,\ybar)$ be a quantifier-free $\cL$-formula with parameters from $M$. Then $\pi^{\cM}_\varphi(n)\leq \pi^{\cN}_\varphi(n)$ for all $n\geq 0$.
\end{proposition}
\begin{proof}
Fix $n\geq 0$ and let $k=\pi^{\cM}_{\varphi}(n)$. Then there are pairwise distinct $\abar_1,\ldots,\abar_n\in M^{\xbar}$ and pairwise distinct subsets $X_1,\ldots,X_k\seq [n]$ such that for all $1\leq j\leq k$, there is some $\bbar_j\in M^{\ybar}$ with $\cM\models\varphi(\abar_i,\bbar_j)$ if and only if $i\in X_j$. Since $\varphi(\xbar,\ybar)$ is quantifier-free, and $\cM$ is a substructure of $\cN$, it follows that for all $1\leq i\leq n$ and $1\leq j\leq k$, $\cN\models\varphi(\abar_i,\bbar_j)$ if and only if $i\in X_j$. Thus $\abar_1,\ldots,\abar_n$ (which are also in $N^{\xbar}$) and $X_1,\ldots,X_k$ witness $\pi^{\cN}_{\varphi}(n)\geq k$.
\end{proof}

\begin{definition}
    Let $\cM$ be an $\cL$-structure.
\begin{enumerate}[(1)]
\item Let $\varphi(\xbar,\ybar)$ be an $\cL$-formula with parameters from $M$. Then $\varphi(\xbar,\ybar)$ is \textbf{NIP in $\cM$} if $\VC(\cS^{\cM}_\varphi)<\infty$. 
\item $\cM$ is called \textbf{NIP} if every bi-partitioned $\cL$-formula $\varphi(\xbar,\ybar)$ is NIP in $\cM$.
\end{enumerate}
\end{definition}

The acronym NIP stands for the ``negation of the independence property" and originates from Shelah's  seminal work \cite{Shelah} on classification of first-order theories.
NIP structures form an important class in model theory and include (among other things) abelian groups, planar graphs, linear orders, dense meet-trees, algebraically closed fields, algebraically closed valued fields, the real field, and the field $\Q_p$. The standard text on NIP structures and theories is Simon's book \cite{Sibook}, which includes further explanations and citations for these examples (in particular, see \cite[Example 2.12]{Sibook}. Other notable examples of NIP structures include ordered abelian groups \cite{GurSch}, the real field with exponentiation \cite{Wilkie}, and finitely generated free groups \cite{Sela}.

\section{Preliminaries on VC-dimension in groups}\label{sec:prelim2}

\subsection{Definitions and basic properties}
Let $G$ be an arbitrary group.  Given subsets $A,B\seq G$, we let $AB=\{ab:a\in A,~b\in B\}$. Given $g\in G$ and $A\seq G$, the translates $gA$ and $Ag$ are defined similarly. We also let $A\inv=\{a\inv:a\in A\}$. When $G$ is abelian, we typically switch to the additive analogues of this notation ($g+A$, $\nv A$, $A+B$, etc.).

Given $A\seq G$, we let $\cS_G(A)$ denote the set system on $G$ consisting of all left translates of $A$. Set  $\VC_G(A)=\VC(\cS_G(A))$. We also set $\pi^G_A=\pi_{\cS_G(A)}$. When the ambient group $G$ is understood, we may omit the superscript $G$ and just write $\pi_A$. If  $\VC_G(A)<\infty$, then $A$ is called a \textbf{VC-set in $G$}.

\begin{example}\label{ex:subgroupVC}
Suppose $H$ is a subgroup of $G$. Then $\VC_G(H)=0$ if $H=G$, and $\VC_G(H)=1$ if $H\neq G$. Moreover, $\pi^G_H(n)=\min\{n+1,[G:H]\}$ for $n\leq |G|$.
\end{example}

\begin{remark}\label{rem:VC-symmetric}
    Our focus on set systems generated by left translates of a fixed set follows conventions established in previous literature. From a qualitative standpoint, this choice does not matter. In particular, one can use VC-duality to show that for any group $G$ and $A\seq G$, if $\VC_G(A)<d$ then the set system of right translates of $G$ has VC-dimension at most $2^d$ (see also \cite[Corollary 3.19]{CPtrip}). It is also easy to show that if $A$ is \emph{symmetric} (i.e., $A=A\inv$) then the set system of left translates of $A$ has the same VC-dimension as the set system of right translates of $A$. More generally, given any set system $\cS$ on $G$, if $\cS\inv=\{A\inv:A\in \cS\}$ then $\pi_{\cS}=\pi_{\cS\inv}$ by Proposition \ref{prop:pullback} and the fact that $x\mapsto x\inv$ is a bijection from $G$ to $G$.
\end{remark}

\begin{remark}\label{rem:generated}
Given a group $G$ and a subset $A\seq G$, when analyzing $\VC_G(A)$ one can assume that $G$ is generated by $A$ with only minimal loss of quantitative information. In fact, if we let $\cS=\{gA: g \in AA\inv\}$, then it is an easy exercise to show that $\VC(\cS)\leq\VC_G(A)\leq\VC(\cS)+1$. More precisely, if $\cS_G(A)$ shatters a set $X\seq G$, then $\cS$ will shatter every proper subset of some left translate of $X$. (This is closely related to \cite[Proposition 4.1(3)]{Sisask}; see also \cite[Remark 3.21]{CPtrip}.)
\end{remark}

\begin{proposition}\label{prop:translateVC}
If $A\seq G$ and $g\in G$ then $\pi^G_A=\pi^G_{gA}=\pi^G_{Ag}$, and so $\VC_G(A)=\VC_G(gA)=\VC_G(Ag)$. 
\end{proposition}
\begin{proof}
    Note that $\cS_G(A)=\cS_G(gA)$, so $\pi^G_A=\pi^G_{gA}$. Now define $f\colon G\to G$ such that $f(x)=xg\inv$. Then $f$ is a bijection and $f\inv(\cS_G(A))=\cS_G(Ag)$ so $\pi^G_A=\pi^G_{Ag}$ by Proposition \ref{prop:pullback}.
\end{proof}

\begin{corollary}\label{cor:VCBA}
Let $\cB$ be the collection of VC-sets in $G$. Then $\cB$ is a Boolean algebra, which contains all subgroups of $G$ and is closed under inversion and under left and right translation.
\end{corollary}
\begin{proof}
By Example \ref{ex:subgroupVC}, Remark \ref{rem:VC-symmetric}, Proposition \ref{prop:translateVC}, $\cB$ contains all subgroups of $G$ and is closed under inversion and under left and right translation. Now suppose $A,B\in\cB$. Then $\cS_G(G\backslash A)=\neg\cS_G(A)$, so $G\backslash A\in \cB$ by Remark \ref{rem:VCcomp}. Moreover, $\cS_G(A\cap B)\seq \cS_G(A)\wedge\cS_G(B)$, so $A\cap B\in\cB$ by Corollary \ref{cor:VCintersection}.
\end{proof}

Finally, we prove a technical lemma that will be used in Section \ref{sec:Heisenberg} in order to make a modest improvement to a specific VC-dimension bound in the context of the Heisenberg group. Roughly speaking, the result says that if a set $A\seq G$ is partitioned into sets $A_1,\ldots,A_k$ by intersecting with the cosets of an index $k$ subgroup, then $\VC_G(A)$ can be controlled by $\VC_G(A_1),\ldots,\VC_G(A_k)$ in a stronger fashion compared to what is given by a general bound for unions of VC-sets (see Remark \ref{rem:fvsg}).

\begin{lemma}\label{lem:coset-union}
Let $H\leq G$ be a subgroup of finite index $k<\infty$. Fix $A\seq G$ and suppose $\VC_G(A)>(n-1)k$ for some integer $n\geq 1$. Then
\[
n\leq \log\left(k\max_{C\in G/H}\pi_{A\cap C}(n)\right).
\]
In particular, if $\VC_G(A\cap C)\leq d$ for all $C\in G/H$, then $n\leq \log(k\mathfrak{C}_d(n))$.
\end{lemma}
\begin{proof}
Since $\VC_G(A)>(n-1)k$, we may fix a set $X\seq G$ shattered by $\cS_G(A)$, with $|X|>(n-1)k$. By pigeonhole, there must be some left coset $C$ of $H$ such that $|X\cap C|\geq n$. For each $S\seq X\cap C$, choose $g_S\in G$ such that $g_SA\cap X=S$. Consider the map from $\cP(X\cap C)$ to $G/H$ sending $S$ to $g_S\inv C$. Since $|\cP(X\cap C)|\geq 2^n$ and $|G/H|=k$, there is some $D\in G/H$ such that the fiber over $D$ has size $m$ for some $m\geq 2^n/k$. So there are pairwise distinct $S_1,\ldots,S_m\seq X\cap C$ such that $C=g_{S_i}D$ for all $1\leq i\leq m$. Now for $1\leq i\leq m$, we have
\[
g_{S_i}(A\cap D)\cap (X\cap C)=g_{S_i}A\cap X\cap C=S_i\cap C=S_i
\]
Since $|X\cap C|\geq n$, we have $\pi_{A\cap D}(n)\geq m\geq 2^n/k$, i.e., $n\leq \log(k\pi_{A\cap D}(n))$.  

This establishes the first claim in the lemma. The second claim  follows using Lemma \ref{lem:SSL}. 
\end{proof}

\begin{remark}\label{rem:fvsg}
Let $H\leq G$ be a subgroup of finite index $k<\infty$, and fix some $A\seq G$. The previous lemma shows that if we have $\VC_G(A\cap C)\leq d$ for all $C\in G/H$, then $\VC_G(A)\leq \mathfrak{g}(d,k)$ where
\[
\mathfrak{g}(d,k)=k(\min\{n\in\N:\log(k\mathfrak{C}_d(n))<n\}-1).
\]
Let us compare this to the bound given by Section \ref{sec:set-systems}. In particular, $A$ is the union of its intersection with each coset of $H$, and so Corollary \ref{cor:VCintersection} and Remark \ref{rem:VCcomp} yield $\VC_G(A)<\mathfrak{f}(d,k)$. The $\mathfrak{g}(d,k)$  bound is better when $d$ is large relative to $k$. For example, one can directly show that the inequality $\mathfrak{g}(d,k)\leq \mathfrak{f}(d,k)-k$ holds for $d>k\geq 2$. Using Remark \ref{rem:taylor} and estimates of the Lambert function (e.g., \cite{ChatLbounds}), one can obtain improved bounds such as $\mathfrak{g}(d,k)\leq 5(2k)^{\frac{\log(e)}{d}}dk$. So if $d\geq O(\frac{\log k}{\log \log k})$, then $\mathfrak{g}(d,k)$ is asymptotically smaller than the lower bound of $dk\log k$ for $\mathfrak{f}(d,k)$ from Remark \ref{rem:fbounds}. 
\end{remark}

\subsection{Examples, revisited}

Recall that the examples from Section \ref{sec:examples} were each set systems on $\R^n$ or $\T^n$. Both of these  sets are also abelian groups, and one sees that the set systems from Section \ref{sec:examples} are  invariant under the respective group operation. Thus we obtain a number of examples of VC-sets in these particular abelian groups. Toward finding VC-sets in other groups, we make the following general observation.

\begin{remark}
    Let $\tau\colon G\to H$ be a  homomorphism between groups. Then for any $B\seq H$, we have $\cS_G(\tau\inv(B))\seq \tau\inv(\cS_H(B))$, and thus $\VC_G(\tau\inv(B))\leq \VC_H(B)$ by Proposition \ref{prop:pullback}.
\end{remark}

A special case of the above situation is the notion of a Bohr set.

\begin{example}[Bohr sets]\label{ex:Bohr}
 Let $G$ be an arbitrary group. Given a real number $\delta>0$ and an integer $k\geq 1$, a \textbf{$(\delta,k)$-Bohr set} in $G$ is a subset of the form $B=\tau\inv(U)$ where $\tau\colon G\to\T^k$ is a homomorphism and $U_\delta\seq \T^k$ is the open identity neighborhood  of radius $\delta$ with respect to the product of the arclength metric on $S^1$ (normalized to $1$). Note that $U_\delta$ is an axis-parallel cube in $\T^k$ (i.e., a member of the system $\cS_c(\T^k)$ defined in Example \ref{ex:box-torus}). Altogether, we conclude from \cite{GLM} that if $B$ is a $(\delta,k)$-Bohr set in $G$, then 
 \[
 \VC_G(B)\leq \VC_{\T^k}(U_\delta)\leq\VC(\cS_c(\T^k))\leq   k\log k+3k\log\log k.
 \]
One might expect to be able to do better since the set of $\T^k$-translates of $U_\delta$ comprises a much smaller set system that \emph{all} axis-parallel cubes. While we do not know of any general improvement to the above bound, it is interesting to note that if $\delta$ is sufficiently small, then one can obtain a better bound using Example \ref{ex:cube-R} and  the fact that translates of $U_\delta$ in $\T^k$ behave more like translates of cubes in $\R^k$. This observation was used by Sisask \cite{Sisask} to bound $\VC_G(B)$ by $2k$ in the case that $G$ is abelian, where $2n$ comes from the more well-known Example \ref{ex:box-R}. Using similar methods, one can show that if $\delta\leq 1/4$ (recall our metric on $\T^k$ defined above), then $\VC_{\T^k}(U_\delta)\leq \VC(\cS_c(\R^k))\leq \lfloor(3k+1)/2\rfloor$. 
\end{example}

\subsection{Motivation: nilprogressions}\label{sec:motivation}

The motivation for our interest in the VC-dimension of generalized progressions  starts with a result from \cite{CPT}, which provides an approximate structure theorem for a subset $A$ of a finite group $G$ with $\VC_G(A)$ bounded by some fixed integer $k$. Roughly speaking, the result says that $A$ can be approximated by a union of translates of a Bohr set of bounded complexity in a normal subgroup of $G$ of bounded index. The methods involved in this proof draw from high-powered model theory, combinatorics, and group theory, and we will not get into further details here (for $G$ abelian, a similar result was obtained independently by Sisask \cite{Sisask} using finitary combinatorial methods). In light of Example \ref{ex:Bohr}, we see that  VC-sets in finite groups are approximated by a special kind of  VC-set. In \cite{CPtrip}, the methods from \cite{CPT} were extended to the setting of a (nonempty) finite subset $A$ of an arbitrary group $G$ with $\VC_G(A)$ and $|AAA|/|A|$ both uniformly bounded (the latter condition is closely related to the notion of an approximate subgroup \cite{BGT, TaoPSE}).  In this case, $A$ can be approximated by a special kind of generalized progression called a ``coset nilprogression" (defined below). However, whether coset nilprogressions are themselves  VC-sets remains open. This leads to the following  question, which was the initial motivation for the work in this paper.

\begin{question}\label{Q:nil-simple}
    Let $G$ be an $s$-nilpotent group and fix $a_1,\ldots,a_k\in G$. Do we have $\VC_G(P(\abar,\bar{N}))\leq O_{k,s}(1)$ for any $\bar{N}\in\N^k$? 
\end{question}

To explain the connection between Question \ref{Q:nil-simple} and the problem alluded to at the start of this subsection, we need the following definitions from the literature on approximate groups \cite{BGT}. Let $G$ be a group. A \emph{nilprogression of step $s$} in $G$ is a generalized progression $P(\abar,\bar{N})$ such that $\abar$ generates an $s$-nilpotent subgroup of $G$. A \emph{coset nilprogression (of step $s$)} is a set of the form $PH$ where $P$ is a nilprogression of step $s$ and $H$ is a finite subgroup of $G$ normalized by $P$. Returning to the discussion above, we are really interested in bounding $\VC_G(C)$ where $C$ is a coset nilprogression and $G$ is arbitrary. However, as we now explain, this sitatuation reduces to Question \ref{Q:nil-simple}. In particular, suppose $C=PH$ is a coset nilprogression of step $s$, with $P$ of rank $k$. Then $H$ is a normal subgroup of $K:=\langle C\rangle $, and one can check that $C/H$ is a nilprogression in $K/H$ of rank $k$ and step $s$. Thus a positive answer to Question \ref{Q:nil-simple} would give $\VC_{K/H}(C/H)\leq O_{k,s}(1)$. Moreover, we have $\VC_{K/H}(C/H)=\VC_K(C)$ (e.g., using Proposition \ref{prop:pullback}) and $\VC_G(C)\leq \VC_K(C)+1$ (by Remark \ref{rem:generated}).

As a matter of fact, we are not currently aware of any counterexample to the following question, which is much more ambitious (and possibly outrageous). 

\begin{question}\label{Q:allGAP}
    Let $G$ be a group and fix $a_1,\ldots,a_k\in G$. Do we have $\VC(\cP_G(\abar))\leq O_k(1)$?
\end{question}

Note that by Remark \ref{rem:generated}, one can assume in Question \ref{Q:allGAP} that $G$ is generated by $a_1,\ldots,a_k$. As explained in the introduction, here we will restrict this question to two very special cases, namely free groups (see Section \ref{sec:free}) and the Heisenberg group over $\Z$  (see Section \ref{sec:Heisenberg}).

\begin{remark}\label{rem:abelian}
    Despite our focus on abelian groups, the reader may have noticed that we have not fully addressed the abelian  (i.e., $s=1$) case of Question \ref{Q:nil-simple}. Given a generalized  progression $P=P(\abar,\bar{N})$ of rank $k$ in an abelian group $G$, if one assumes a certain local freeness condition called ``properness" (specifically, $|P-P|=\prod_i (4N_i+1)$), then the set system of translates of $P$ behaves enough like Example \ref{ex:box-R} to recover $\VC_G(P)\leq 2k+1$. A detailed argument (for $G=\Z$) is given by Sisask in \cite[Proposition 4.11]{Sisask}. Sisask further asserts that VC-dimension bounds on Bohr sets (see Example \ref{ex:Bohr}) can be used to relax the properness condition, though details are not given. Thus it seems that even in the abelian case, a complete answer the kinds of questions we are asking does not appear in existing literature. On the other hand, it is also worth pointing out than an analogue of properness exists for  nilprogressions, which goes by the name ``$c$-normal form" where $c$ is an integer parameter (see \cite[Definition 2.6]{BGT}). Thus  Question \ref{Q:nil-simple} could become more tractable if one were to further assume that $P(\abar,\bar{N})$ is in $c$-normal form, and allow $\VC_G(P(\abar,\bar{N}))$ to be bounded in terms of $k$, $s$, and $c$. That being said, in light of our results in Section \ref{sec:Heisenberg}, we note that nilprogressions in the Heisenberg group generated by the standard generators are not in uniformly bounded normal form (see \cite[Example 11]{BGT}).
\end{remark}

\section{Generalized Progressions in the Heisenberg group}\label{sec:Heisenberg}

Before discussing the Heisenberg group, we first provide a more rigorous description of generalized progressions in terms of some specific notation that will be useful in this section. 

\begin{notation}\label{not:GAP}
Fix a group $G$ and $a_1,\ldots,a_k\in G$. We let $\Sigma(\abar)$ denote the set of formal finite strings (or ``words") with coordinates from $\{a_1,\ldots,a_k,a_1\inv,\ldots,a_k\inv\}$. Suppose $w$ is a word in $\Sigma(\abar)$. Let $\ulcorner w\urcorner$ denote the unique element of $G$ obtained by multiplying the coordinates of $w$. Given $1\leq i\leq k$, let $n^+_{a_i}(w)$ denote the number of occurrences of $a_i$ in $w$, and let $n^-_{a_i}(w)$ denote the number of occurrences of $a_i\inv$ in $w$. Set $n_{a_i}(w)=n^+_{a_i}(w)+n^-_{a_i}(w)$. Finally given natural numbers $N_1,\ldots,N_k\in\N$, define
\[
W(\abar,\bar{N}) = \{w\in \Sigma(\abar):n_{a_i}(w)\leq N_i\text{ for all }1\leq i\leq k\}.
\]
Then $P(\abar,\bar{N})$ is, by definition, the set $\{\ulcorner w\urcorner:w\in W(\abar,\bar{N})\}$.
\end{notation}

Now recall that the (integer) Heisenberg group consists of all $3\times 3$ upper-triangular  matrices with $1$ on the diagonal and arbitrary integers above the the diagonal (as a group under standard matrix multiplication). To simplify notation, we will identify the integer matrix $\left(\begin{smallmatrix}
1 & x & z\\
0 & 1 & y\\
0 & 0 & 1
\end{smallmatrix}\right)$
 with the tuple $(x,y,z)\in\Z^3$. Consequently, the Heisenberg group can be described as $\H=(\Z^3,\ast)$ where
\[
(x,y,z)\ast(x',y',z')=(x+x',y+y',z+z'+xy').
\]

Let $A,B,C\in \H$ denote $(1,0,0)$, $(0,1,0)$, and $(0,0,1)$ respectively. We have the following basic observations which, in particular, identify $\H$ as the free $2$-step nilpotent group on two generators.

\begin{remark}\label{rem:Hid}$~$
\begin{enumerate}[$(a)$]
\item If $a,b,c\in\Z$, then $(a,b,c)=B^bA^aC^c$.
\item If $a,b\in\Z$, then  $A^aB^b=B^bA^aC^{ab}$.
\end{enumerate}
In particular, $\H$ is generated by $A$ and $B$, and we have $C=[A,B]$.
\end{remark}

Throughout this section, we will work with generalized progressions in $\H$ generated by $A$ and $B$. So given $N_1,N_2\geq 0$, we let $P(N_1,N_2)=P(A,B;N_1,N_2)$.  Our first main goal is to obtain a very precise description of  $P(N_1,N_2)$, which is given in Theorem \ref{thm:PNNexp} below. 

We now specialize Notation \ref{not:GAP} to $\H$ with generators $A$ and $B$. Let $\Sigma$ denote the set of formal strings in the symbols $A,B,A\inv,B\inv$. Recall  the notation $n^s_X(w)$ for $w\in\Sigma$, $X\in\{A,B\}$ and $s\in\{+,-\}$, as well as $n_X(w)=n^+_X(w)+n^-_X(w)$. 

Next, we describe an algorithm for reducing a word $w$ in $\Sigma$ to an element $\ulcorner w\urcorner=(a,b,c)$ in $\H$. 

\begin{algorithm}\label{alg}
Fix a word $w\in\Sigma$. We generate a finite sequence of pairs in $\Sigma\times\Z$ as follows. Let $(w_0,j_0)=(w,0)$. Then given  $(w_n,j_n)$, do the following.
    Scan $w_n$ from right to left for an occurrence of $A^{\delta}$ to the left of $B^{\epsilon}$, where $\delta,\epsilon\in\{1,-1\}$. If there is no such occurrence then halt and output the sequence $(w_0,j_0),\ldots,(w_n,j_n)$. Otherwise,  pick the right-most of such occurrences and write $w_n=u_n A^{\delta}B^{\epsilon}v_n$. Let $w_{n+1}=u_nB^{\epsilon}A^{\delta}v_n$ and $j_{n+1}=j_n+\delta\epsilon$. Then repeat the previous process with  $(w_{n+1},j_{n+1})$.  

Now, given $w\in\Sigma$, suppose $(w_0,j_0),\ldots,(w_n,j_n)$ is the resulting output of the above algorithm. Then, for each $0\leq i<n$, we we have the following properties.
\begin{enumerate}[$(1)$]
\item $|j_i-j_{i+1}|=1$.
\item $n^{s}_X(w_i)=n^s_X(w_{i+1})$ for any $s\in\{+,-\}$ and $X\in\{A,B\}$. 
\item $\ulcorner w_i[A,B]^{j_i}\urcorner = \ulcorner w_{i+1}[A,B]^{j_{i+1}}\urcorner$ (by Remark \ref{rem:Hid}$(b)$).
\end{enumerate}
\end{algorithm}

The previous algorithm yields an explicit description of the word reduction map from $\Sigma$ to $\H$.

\begin{proposition}
    \label{prop:heisenberg_formula}
    Suppose $w=E_1^{\epsilon_1}\cdots E_k^{\epsilon_k}\in \Sigma$, where $E_1,\ldots, E_k\in \{A,B\}$ and $\epsilon_1,\ldots, \epsilon_k \in \{1, -1\}$. Then $\ulcorner w\urcorner=(a,b,c)$ where:
    \begin{enumerate}[\hspace{5pt}$\ast$]
    \item $a=n^+_A(w)-n^-_A(w)=\sum_{E_i=A}\epsilon_i$, 
    \item $b=n^+_B(w)-n^-_B(w)=\sum_{E_j=B}\epsilon_j$, and
    \item $c = \sum_{E_i=A,\,E_j=B,\, i<j}\epsilon_i\epsilon_j$.
    \end{enumerate}
\end{proposition}
\begin{proof}
Let $(w_0,j_0),\ldots,(w_n,j_n)$ be the output of Algorithm \ref{alg} with input $w$. Then $w_n$ is of the form $B^{\delta_1}\cdots B^{\delta_\ell}A^{\delta'_{1}}\cdots A^{\delta'_{m}}$ for some $\ell,m\geq 0$ and $\delta_i,\delta'_j\in \{1,-1\}$. Therefore 
\[
\textstyle\ulcorner w_n[A,B]^{j_n}\urcorner = \left(  \sum_{i=1}^{m}\delta_i,\sum_{i=1}^\ell \delta'_i,\,j_n\right)=\left(n^+_A(w_n)-n^-_A(w_n),n^+_B(w_n)-n^-_B(w_n),j_n\right).
\]
Combining this with  $(2)$ and $(3)$ above, we conclude $\ulcorner w\urcorner = \left(n^+_A(w)-n^-_A(w),n^+_B(w)-n^-_B(w),j_n\right)$.  Finally, it follows from the algorithm  that $j_n=\sum_{E_i=A,\,E_j=B,\, i<j}\epsilon_i\epsilon_j$.
\end{proof}

We  record a few easy observations that will be used later.

\begin{corollary}\label{cor:general_properties_of_algorithm}
    Fix $w\in\Sigma$ and suppose $\ulcorner w\urcorner =(a,b,c)\in\H$.
    \begin{enumerate}[$(a)$]
    \item $|a|\leq n_A(w)$ and $|b|\leq n_B(w)$. 
        \item $n_A(w)+a=2n^+_A(w)$ and $n_B(w)+b=2n^+_B(w)$.
        \item If $w^*$ denotes the reverse word obtained from $w$, then $\ulcorner w^*\urcorner=(a,b,ab-c)$.
    \end{enumerate}
\end{corollary}
\begin{proof}
    Parts $(a)$ and $(b)$ follow from the fact that $a=n^+_A(w)-n^-_A(w)$ and $b=n^+_B(w)-n^-_B(w)$. 

    Part $(c)$. Suppose $w=E_1^{\epsilon_1}\cdots E_k^{\epsilon_k}$, where $E_1,\ldots, E_k\in \{A,B\}$ and $\epsilon_1,\ldots, \epsilon_k \in \{1, -1\}$. Then by Proposition \ref{prop:heisenberg_formula}, $\ulcorner w^*\urcorner=(a,b,c')$ where $c'=\sum_{E_i=A,\,E_j=B,\,i>j}\epsilon_i\epsilon_j$. Therefore
    \[
    c+c'=\sum_{E_i=A,\,E_j=B}\epsilon_i\epsilon_j=\left(\sum_{E_i=A}\epsilon_i\right)\left(\sum_{E_j=B}\epsilon_j\right)=ab.
    \]
    So $c'=ab-c$, as desired.
\end{proof}

Next we prove the main technical lemma.

\begin{lemma}\label{lem:corner-tech}
For any $w\in\Sigma$, if $\ulcorner w\urcorner = (a,b,c)$ and $a,b\geq 0$, then $c\leq n^+_A(w)n^+_B(w)$. 
\end{lemma}
\begin{proof}
We proceed by induction on the length of $w$, where the base case (the empty word) is immediate. Fix a word $w\in\Sigma$ and assume the result for all words in $\Sigma$ of length strictly less than the length of $w$. Suppose $\ulcorner w\urcorner =(a,b,c)$ with $a,b\geq 0$. 

\noindent\emph{Case 1.} $n^-_A(w)=0$.

In this case, we have $n_A(w)=n^+_A(w)$. Let $w=E_1^{\epsilon_1}\cdots E_k^{\epsilon_k}$ for some $E_i\in \{A,B\}$ and $\epsilon_i\in \{1,-1\}$. By assumption of this case, $\epsilon_i=1$ whenever $E_i=A$. So by Proposition \ref{prop:heisenberg_formula}, 
\[
c=\sum_{E_i=A,~E_j=B,~i<j}\epsilon_j \leq \sum_{E_i=A,~E_j=B,~\epsilon_j=1}\epsilon_j=\sum_{E_i=A}n_B^+(w)= n^+_A(w)n^+_B(w).
\]

\noindent\emph{Case 2.} $n^-_A(w)>0$.

In this case, we must also have $n^+_A(w)>0$ since $a\geq 0$. So we can write $w=w_1A^\delta w_2 A^\epsilon w_3$ for some distinct $\delta,\epsilon\in\{1,-1\}$ and $w_1,w_2,w_3\in\Sigma$. For $s\in\{+,-\}$ and $t\in \{1,2,3\}$, let $b^s_t=n^s_B(w_t)$, and set $b_t=b^+_t-b^-_t$. Note that $b_1+b_2+b_3=n^+_B(w)-n^-_B(w)=b\geq 0$. Set $w'=w_1w_2w_3$. Using Proposition \ref{prop:heisenberg_formula}, one can check that  $\ulcorner w'\urcorner = (a,b,c-\delta b_2)$. 
Therefore, by  induction, we have
\[
c-\delta b_2\leq n^+_A(w')n^+_B(w').
\]
So to prove $c\leq n^+_A(w)n^+_B(w)$, it suffices to show
\[
n^+_A(w')n^+_B(w')+\delta b_2\leq n^+_A(w)n^+_B(w).\tag{$1$}
\]
Using $n^+_A(w)=n^+_A(w')+1$ and $n^+_B(w)=n^+_B(w')=b^+_1+b^+_2+b^+_3$, we see that $(1)$ is equivalent to
\[
(\delta-1)b^+_2\leq b^+_1+\delta b^-_2+b^+_3.\tag{$2$}
\]
If $\delta=1$ then $(2)$ holds trivially since each $b^s_t$ is nonnegative. So assume $\delta=-1$. Then $(2)$ becomes
\[
-2b^+_2\leq b^+_1-b^-_2+b^+_3.\tag{$3$}
\]
Recall that $b_1+b_2+b_3\geq 0$, i.e., $b^+_1-b^-_2+b^+_3\geq b^-_1-b^+_2+b^-_3$. So to prove $(3)$, it suffices to show
\[
-2b^+_2\leq b^-_1-b^+_2+b^-_3. \tag{$4$}
\]
Now $(4)$ is equivalent to $0\leq b^-_1+b^+_2+b^-_3$, which holds trivially since each $b^s_t$ is nonnegative.
\end{proof}

Given $N_1,N_2\geq 0$,  let $W(N_1,N_2)=\{w\in\Sigma:n_A(w)\leq N_1,~n_B(w)\leq N_2\}$. The following is an immediate consequence of Corollary \ref{cor:general_properties_of_algorithm}$(b)$.

 \begin{remark}\label{rem:even}
   For any $N_1,N_2\geq 0$, if $w\in W(N_1,N_2)$ then $n^+_A(w)\leq \left\lfloor \frac{N_1+a}{2}\right\rfloor$ and $n^+_B(w)\leq \left\lfloor \frac{N_2+b}{2}\right\rfloor$. 
 \end{remark}
 
 Next we prove two properties related to the  symmetry and convexity of $P(N_1,N_2)$.
 
 \begin{lemma}\label{lem:sym-conv}
 Fix $a,b,c\in\Z$ and $N_1,N_2\geq 0$.
 \begin{enumerate}[$(a)$]
 \item $P(N_1,N_2)$ contains either all or none of $(a,b,c)$, $(-a,-b,c)$, $(-a,b,-c)$, and $(a,-b,-c)$.
 \item Suppose $(a,b,c)\in P(N_1,N_2)$. Then $(a,b,c')\in P(N_1,N_2)$ for all $0\leq c'\leq c$ if $c>0$, and for all $c\leq c'\leq 0$ if $c<0$.
 \end{enumerate}
 \end{lemma}
 \begin{proof}
Part $(a)$. Since $a$, $b$, and $c$ are arbitrary, it suffices to assume $(a,b,c)\in P(N_1,N_2)$, and show $(-a,-b,c),(-a,b,-c)\in P(N_1,N_2)$. Fix $w\in W(N_1,N_2)$ such that $\ulcorner w\urcorner=(a,b,c)$. Let $u$ be the word formed by replacing all $A$ with $A^{-1}$ (and vice versa), and all $B$ with $B^{-1}$ (and vice versa). Let $v$ be the word formed by replacing all $A$ with $A^{-1}$ (and vice versa). Then it is clear that $v,w\in W(N_1,N_2)$. By Proposition \ref{prop:heisenberg_formula}, $\ulcorner u\urcorner=(-a,-b,c)$ and $\ulcorner v\urcorner=(-a,b,-c)$.  

Part $(b)$. Choose $w\in \Sigma$ such that $\ulcorner w\urcorner=(a,b,c)$, and let $(w_0,j_0),\ldots,(w_n,j_n)$ be the output of Algorithm \ref{alg}. Fix $c'$ is as above. Since $j_0=0$, $j_n=c$, and $|j_i-j_{i+1}|=1$ for all $0\leq i<n$,  there is some $0\leq m\leq n$ satisfying $j_{m}=c-c'$. One can then check that , by running Algorithm \ref{alg} with initial word $w_m$, we obtain $\ulcorner w_m\urcorner=(a,b,c')$. Moreover, since $n_A(w)=n_A(w_m)$ and $n_B(w)=n_B(w_m)$, we have $w_m\in W(N_1,N_2)$ and hence $(a,b,c')\in P(N_1,N_2)$.
\end{proof}

We can now give an explicit description of the generalized progression $P(N_1,N_2)$ in $\H$.

\begin{theorem}\label{thm:PNNexp}
If $N_1,N_2\geq 0$ and $a,b,c\in\Z$, then $(a,b,c)\in P(N_1,N_2)$ if and only if one of the following cases holds.
\begin{enumerate}[$(1)$]
\item $0\leq a\leq N_1$, $0\leq b\leq N_2$, and
\[
ab-\left\lfloor\frac{N_1+a}{2}\right\rfloor\left\lfloor\frac{N_2+b}{2}\right\rfloor\leq c\leq \left\lfloor\frac{N_1+a}{2}\right\rfloor\left\lfloor\frac{N_2+b}{2}\right\rfloor.
\]
\item $-N_1\leq a<0$, $-N_2\leq b<0$, and 
\[
ab-\left\lfloor\frac{N_1-a}{2}\right\rfloor\left\lfloor\frac{N_2-b}{2}\right\rfloor\leq c\leq \left\lfloor\frac{N_1-a}{2}\right\rfloor\left\lfloor\frac{N_2-b}{2}\right\rfloor.
\]
\item $0\leq a\leq N_1$, $-N_2\leq b<0$, and
\[
-\left\lfloor\frac{N_1+a}{2}\right\rfloor\left\lfloor\frac{N_2-b}{2}\right\rfloor\leq c\leq ab+ \left\lfloor\frac{N_1+a}{2}\right\rfloor\left\lfloor\frac{N_2-b}{2}\right\rfloor.
\]
\item $-N_1\leq a<0$, $0\leq b\leq N_2$, and
\[
-\left\lfloor\frac{N_1-a}{2}\right\rfloor\left\lfloor\frac{N_2+b}{2}\right\rfloor\leq c\leq ab+ \left\lfloor\frac{N_1-a}{2}\right\rfloor\left\lfloor\frac{N_2+b}{2}\right\rfloor.
\]
\end{enumerate}
\end{theorem}
\begin{proof}
Note that if $(a,b,c)\in P(N_1,N_2)$ then $|a|\leq N_1$ and $|b|\leq N_2$ by Corollary \ref{cor:general_properties_of_algorithm}$(a)$. Together with Lemma \ref{lem:sym-conv}$(a)$, we see that it suffices to fix $0\leq a\leq N_1$ and $0\leq b\leq N_2$, and show that $(a,b,c)\in P(N_1,N_2)$ if and only if 
\[
ab-c^*\leq c\leq c^*,
\]
where $c^*=\left\lfloor\frac{N_1+a}{2}\right\rfloor\left\lfloor\frac{N_2+b}{2}\right\rfloor$.
Let $I=\{c\in \Z:(a,b,c)\in P(N_1,N_2)\}$. Then $I$ is an interval in $\Z$ by Lemma \ref{lem:sym-conv}$(b)$. Also, by Corollary \ref{cor:general_properties_of_algorithm}$(c)$, we have $c\in I$ if and only if $ab-c\in I$. In particular, $\min I=ab-\max I$. So it suffices to show $c^*=\max I$.

First we show $c^*\leq \max I$. Define the word
  \[
  w = B^{b-\lfloor (N_2+b)/2\rfloor}A^{\lfloor (N_1+a)/2\rfloor}B^{\lfloor (N_2+b)/2\rfloor}A^{a-\lfloor (N_1+a)/2\rfloor}.
  \]
  Then $\ulcorner w\urcorner =(a,b,c^*)$ by Remark \ref{rem:Hid}$(b)$, and one can check that $w\in W(N_1,N_2)$. So we have $(a,b,c^*)\in P(N_1,N_2)$, i.e., $c^*\in I$.
  
  Finally, we show $\max I\leq c^*$. Fix $c\in I$, and choose $w\in W(N_1,N_2)$ such that $\ulcorner w\urcorner =(a,b,c)$. Then by Lemma \ref{lem:corner-tech} and Remark \ref{rem:even}, we have $c\leq n^+_A(w)n^+_B(w)\leq c^*$, as desired. 
\end{proof}

Now let $\cP_{\H}=\cP_{\H}(A,B)$, and recall (from Definition \ref{def:GAP}) that $\cP_{\H}$ consists of all left translates of all generalized progressions in $H$ generated by $A$ and $B$. We will first use Theorem \ref{thm:PNNexp} to give a short (and ineffective) model-theoretic proof that $\VC(\cP_{\H})$ is finite, followed by a more detailed proof yielding explicit bounds.

\begin{corollary}\label{cor:VCP}
$\VC(\cP_{\H})$ is finite. 
\end{corollary}
\begin{proof}
Consider the first-order structure $\mathcal{Z}=(\Z,+,\cdot,<,0,1,E)$ consisting of the ring of integers expanded by a unary predicate $E$ for the set of even integers. Note that the group operation in $\H$ (identified with $\Z^3$) is quantifier-free definable in $\mathcal{Z}$ (even without $E$). Combined with Theorem \ref{thm:PNNexp}, it follows that there is a quantifier-free formula $\varphi(x_1,x_2,x_3;y_1,y_2,y_3,y_4,y_5)$ such that, for any $a,b,c\in\Z$, and $N_1,N_2\geq 0$, $\varphi(x_1,x_2,x_3;a,b,c,N_1,N_2)$ defines the left translate $(a,b,c)\ast P(N_1,N_2)$ (an explicit description of this formula is given in the proof of Theorem \ref{thm:VCH-explicit} below). So finiteness of $\VC(\cP_H)$ follows from the fact that every quantifier-free bi-partitioned formula in $\mathcal{Z}$ is NIP. (Without $E$ this follows from Proposition \ref{prop:substructure} and the fact that the real ordered field is NIP. It is then a simple exercise to show any formula of the form $E(p(\xbar,\ybar))$, with $p\in \Z[\xbar,\ybar]$,  is NIP.)
\end{proof}

Next we will reprove the above corollary in more detail that produces an explicit bound on $\VC(\cP_{\H})$. This bound is  certainly not optimal, however the proof will highlight some interesting ingredients that likely would need to be better understood in order to determine the exact value of $\VC(\cP_{\H})$. In particular, we use the following result of Karpinski and Macintyre \cite{KarMac} on the VC-dimension of semialgebraic families in $\R^n$. 

\begin{theorem}[Karpinski-Macintyre]\label{thm:KarMac}
Suppose $\tau_1(\xbar,\ybar),\ldots,\tau_s(\xbar,\ybar)$ are polynomials over $\R$, with $\xbar=(x_1,\ldots,x_k)$ and $\ybar=(y_1,\ldots,y_\ell)$. Let $d$ be the maximum $\ybar$-degree of $\tau_1,\ldots,\tau_s$. Suppose $\zeta(\xbar,\ybar)$ is a first-order formula obtained as a Boolean combination of atomic relations of the form $\tau_i(\xbar,\ybar)>0$ or $\tau_i(\xbar,\ybar)=0$ where $1\leq i\leq s$. Set  $\mathcal{R}=(\R,+,\cdot,<,0,1)$. Then 
\[
\pi^{\cR}_{\zeta}(n)\leq d(2d-1)^{\ell-1}\sum_{i=0}^\ell 2^i{sn\choose i}.
\]
\end{theorem}
\begin{proof}[Explanation]
This theorem is not stated in \cite{KarMac} explicitly in this form, but follows from the proof of Theorem 2 there (and is actually a special case of that theorem since Karpinski and Macintyre work more generally with $\tau_i$ being $C^\infty$ functions satisfying certain properties). In particular, the proof starts by fixing an integer $V$ such that $\pi^{\cR}_{\zeta}(V)=2^V$, and then bounds $2^V$ by the number of connected components of the complement in $\R^\ell$ of the union of certain  submanifolds. Using a theorem of Warren, this latter number is then bounded by $B\sum_{i=0}^\ell 2^i{sV\choose i}$ where $B$ is a fixed parameter (that we discuss further below). In reading these steps, it is clear that one can replace $V$ by an arbitrary integer $n$ and $2^V$ by $\pi^{\cR}_{\zeta}(n)$. So we obtain $\pi^{\cR}_{\zeta}(n)\leq B\sum_{i=0}^\ell 2^i{sn\choose i}$. Now let $\cF$ be the family of polynomials of the form $\tau_i(\abar,\ybar)-\epsilon$ where $1\leq i\leq s$, $\abar\in\R^k$, and $\epsilon\in\R$. Following Section 1 of \cite{KarMac}, the parameter $B$ must satisfy the property that  for any  $\sigma_1(\ybar),\ldots,\sigma_\ell(\ybar)\in\cF\cup\{0\}$, the real algebraic variety $W$  in $\R^\ell$ defined by $\bigwedge_{j=1}^\ell \sigma_j(\ybar)=0$
has at most $B$ connected components. In other words, $B$ must bound the $0^{\textnormal{th}}$ Betti number of  $W$. In \cite{Milnor}, Milnor provides an upper bound of $d(2d-1)^{\ell-1}$ for the sum of \emph{all} Betti numbers of such a variety.
\end{proof}

\begin{theorem}\label{thm:VCH-explicit}
$\VC(\cP_{\H})\leq 267$. Moreover, for any $N_1,N_2\geq 0$, $\VC_{\H}(P(N_1,N_2))\leq 140$.
\end{theorem}
\begin{proof}
To obtain these bounds, we will need a more explicit description of the formula $\varphi$ referred to in the proof of Corollary \ref{cor:VCP}. So recall that we work in the structure $\mathcal{Z}=(\Z,+,\cdot,<,0,1,E)$ where $E$ is a predicate for the even integers. We first use Theorem \ref{thm:PNNexp} to construct a quantifier-free formula $\psi(x_1,x_2,x_3;z_1,z_2)$ such that for any $N_1,N_2\geq 0$, $\psi(x_1,x_2;N_1,N_2)$ defines $P(N_1,N_2)$ (in $\Z^3$).

Toward this end,  let  $\theta(x_1,x_2,x_3;z_1,z_2,t_1,t_2)$ be the  formula obtained as the disjunction of the following four formulas:
\begin{enumerate}[$(1)$]
\item $0\leq x_1\leq z_1$, $0\leq x_2\leq z_2$, and 
\[
4x_1x_2-(z_1+x_1-t_1)(z_2+x_2-t_2)\leq 4x_3\leq (z_1+x_1-t_1)(z_2+x_2-t_2).
\]
\item  $0\leq x_1\leq z_1$, $-z_2\leq x_2<0$, and 
\[
-(z_1+x_1-t_1)(z_2-x_2-t_2)\leq 4x_3\leq 4x_1x_2+(z_1+x_1-t_1)(z_2-x_2-t_2).
\]
\item  $-z_1\leq x_1<0$, $0\leq x_2\leq z_2$, and 
\[
-(z_1-x_1-t_1)(z_2+x_2-t_2)\leq 4x_3\leq 4x_1x_2+(z_1-x_1-t_1)(z_2+x_2-t_2).
\]
\item $-z_1\leq x_1<0$, $-z_2\leq x_2<0$, and 
\[
4x_1x_2-(z_1-x_1-t_1)(z_2-x_2-t_2)\leq 4x_3\leq (z_1-x_1-t_1)(z_2-x_2-t_2).
\]
\end{enumerate}
Next let $\rho(x_1,x_2;z_1,z_2,t_1,t_2)$ be $E(z_1+x_1-t_1)\wedge E(z_2+x_2-t_2)$. Finally, let $\psi(x_1,x_2,x_3;z_1,z_2)$ be the formula
\[
\bigvee_{(\epsilon_1,\epsilon_2)\in \{0,1\}^2}\rho(x_1,x_2;t_1,t_2,\epsilon_1,\epsilon_2)\wedge \theta(x_1,x_2,x_3;z_1,z_2,\epsilon_1,\epsilon_2).
\]
One can check that $\psi(x_1,x_2,x_3;z_1,z_2)$ has the required properties. 

Next we use the fact that the group operation $\ast$ in $\H$ is (quantifier-free) definable in $\mathcal{Z}$ to obtain a quantifier-free formula  $\varphi(x_1,x_2,x_3;y_1,y_2,y_3,y_4,y_5)$ such that, for any $a,b,c\in\Z$, and $N_1,N_2\geq 0$, $\varphi(x_1,x_2,x_3;a,b,c,N_1,N_2)$ defines the left translate $(a,b,c)\ast P(N_1,N_2)$. In particular, this will yield $\cP_{\H}\seq\cS^{\mathcal{Z}}_\varphi$. Explicitly, let $\varphi(\xbar;\ybar)$ be
\[
\psi(x_1-y_1,x_2-y_2,x_3-y_1(x_2-y_2)-y_3;y_4,y_5).
\]
Note that if we let $\hat{\rho}(x_1,x_2;\ybar,t_1,t_2)$ be the formula $\rho(x_1-y_1,x_2-y_2;y_4,y_5,t_1,t_2)$ and $\hat{\theta}(x_1,x_2,x_3;\ybar,t_1,t_2)$ be the formula $\theta(x_1-y_1,x_2-y_2,x_3-y_1(x_2-y_2)-y_3;y_4,y_5,t_1,t_2)$, then $\varphi(\xbar;\ybar)$  has the form
\[
\bigvee_{(\epsilon_1,\epsilon_2)\in \{0,1\}^2}\hat{\rho}(\xbar;\ybar,\epsilon_1,\epsilon_2)\wedge \hat{\theta}(\xbar;\ybar,\epsilon_1,\epsilon_2).\tag{$\dagger$}
\]

 Our next goal is to compute an upper bound for $\pi^{\cZ}_\varphi$.  This will be done in several steps.\medskip
 
 \noindent\emph{Claim.} Fix $\epsilon_1,\epsilon_2\in\{0,1\}$.
 \begin{enumerate}[$(a)$]
 \item $\pi^{\cZ}_{\hat{\rho}(\xbar;\ybar,\epsilon_1,\epsilon_2)}(n)\leq 4$. 
 \item $\pi^{\cZ}_{\hat{\theta}(\xbar;\ybar,\epsilon_1,\epsilon_2)}(n)\leq 162\sum_{i=0}^5 2^i{14n\choose i}$. 
 \end{enumerate}
 \noindent\emph{Proof.} For part $(a)$, one only needs to note that any instance of $\hat{\rho}(x_1,x_2;\ybar,\epsilon_1,\epsilon_2)$ defines a coset of $2\Z\times 2\Z$ in $\Z\times \Z$. Thus $\cS^{\mathcal{Z}}_{\hat{\rho}(\xbar;\ybar,\epsilon_1,\epsilon_2)}$ consists of the four cosets, which yields the claim (recall Example \ref{ex:subgroupVC}). For part $(b)$, we will use Theorem \ref{thm:KarMac}. First however, we need to move from $\mathcal{Z}$ to $\mathcal{R}=(\R,+,\cdot,<,0,1)$. In particular, let $\zeta(x_1,x_2,x_3;\ybar)$  denote $\hat{\theta}(x_1,x_2,x_3;\ybar,\epsilon_1,\epsilon_2)$. Note that $\zeta(\xbar;\ybar)$ does not involve the symbol $E$, so $\pi^{\cZ}_\zeta=\pi^{\cZ_0}_\zeta$ where $\cZ_0=(\Z,+,\cdot,<,0,1)$. Since $\cZ_0$ is a substructure of $\mathcal{R}=(\R,+,\cdot,<,0,1)$, we have $\pi^{\cZ}_\zeta\leq \pi^{\cR}_\zeta$ by Proposition \ref{prop:substructure}. So to prove $(b)$, we will check that $\zeta(\xbar;\ybar)$ can be obtained as in the setting of Theorem \ref{thm:KarMac} with parameters $k=3$, $\ell=5$, $d=2$, and $s=14$. To do this, first define the following polynomials:
 \begin{center}
\begin{tabular}{l@{\hspace{40pt}}l}
$\sigma_1(\xbar,t_1,t_2)=x_1$ & $\sigma_7(\xbar,t_1,t_2)=4x_3-(4x_1x_2-(t_1+x_1-\epsilon_1)(t_2+x_2-\epsilon_2))$\\
$\sigma_2(\xbar,t_1,t_2)=x_2$ & $\sigma_8(\xbar,t_1,t_2)= 4x_3-(t_1+x_1-\epsilon_1)(t_2+x_2-\epsilon_2)$\\
 $\sigma_3(\xbar,t_1,t_2)=x_1-t_1$ & $\sigma_{9}(\xbar,t_1,t_2)=4x_3+(t_1+x_1-\epsilon_1)(t_2-x_2-\epsilon_2)$\\
$\sigma_4(\xbar,t_1,t_2)=x_1+t_1$ & $\sigma_{10}(\xbar,t_1,t_2)= 4x_3-(4x_1x_2+(t_1+x_1-\epsilon_1)(t_2-x_2-\epsilon_2))$\\
$\sigma_5(\xbar,t_1,t_2)=x_2-t_2$ &  $\sigma_{11}(\xbar,t_1,t_2)=4x_3+(t_1-x_1-\epsilon_1)(t_2+x_2-\epsilon_2)$\\
 $\sigma_6(\xbar,t_1,t_2)=x_2+t_2$ &  $\sigma_{12}(\xbar,t_1,t_2)= 4x_3-(4x_1x_2+(t_1-x_1-\epsilon_1)(t_2+x_2-\epsilon_2))$\\
  & $\sigma_{13}(\xbar,t_1,t_2)=4x_3-(4x_1x_2-(t_1-x_1-\epsilon_1)(t_2-x_2-\epsilon_2))$\\
  & $\sigma_{14}(\xbar,t_1,t_2)= 4x_3-(t_1-x_1-\epsilon_1)(t_2-x_2-\epsilon_2)$
\end{tabular}
\end{center}
For each $1\leq i\leq 14$, let 
\[
\tau_i(\xbar,\ybar)=\sigma_i(x_1-y_1,x_2-y_2,x_3-y_1(x_2-y_2)-y_3,y_4,y_5).
\]
Then $\zeta(\xbar;\ybar)$ is a Boolean combination of  relations of the form $\tau_i>0$ and $\tau_i=0$ for $1\leq i\leq 14$. Note also that the $\ybar$-degree of each $\tau_i$ is at most $2$. So altogether, we can apply Theorem \ref{thm:KarMac} to $\zeta(\xbar,\ybar)$ with the choice of parameters stated above.\hfill$\dashv_{\text{\scriptsize{claim}}}$\medskip

Now let $\cS=\cS^{\mathcal{Z}}_\varphi$. We obtain a bound for $\pi_{\cS}=\pi^{\cZ}_\varphi$ by applying Proposition \ref{prop:intersection} to the expression in $(\dagger)$ above. Using the Claim, this yields
\[
\pi_{\cS}(n)\leq \left(648\sum_{i=0}^5 2^i{14n\choose i}\right)^4.
\]
Recall that if $\log(\pi_{\cS}(n))<n$ then $\VC(\cS)\leq n-1$. Thus  if $4\log(648\sum_{i=0}^5 2^i{14n\choose i})<n$ then $\VC(\cS)\leq n-1$. One can check that the first inequality holds when $n=268$. Therefore $\VC(\cS)\leq 267$. Since $\cP_{\H}\seq\cS$, we have $\VC(\cP_{\H})\leq 267$. This establishes the first statement of  the theorem. 

For the second statement, fix $N_1,N_2\geq 0$ and set $P=P(N_1,N_2)$. We can obtain a bound on $\VC_{\H}(P)$ by following the same steps as above while plugging in $N_1$ for $y_3$ and $t_1$, and $N_2$ for $y_4$ and $t_2$. In the corresponding version of the claim, we obtain the same bound for part $(a)$, while in part $(b)$ we use $\ell=3$ (instead of $5$) and obtain the bound $\pi^{\cZ}_{\theta'}(n)\leq 18\sum_{i=0}^3 2^i{14n\choose i}$ where  $\theta'$ is any formula of the form $\hat{\theta}(\xbar;y_1,y_2,y_3,N_1,N_2,\epsilon_1,\epsilon_2)$ for some $\epsilon_1,\epsilon_2\in\{0,1\}$. If we follow the same argument, this yields $\log(\pi^{\H}_P(n))\leq 4\log(72\sum_{i=0}^3 2^i{14n\choose i})$, which in turn yields $\VC_{\H}(P)\leq 160$. However, we can use Lemma \ref{lem:coset-union} to do a little better. In particular, the above argument shows that if $C$ is a coset of $2\Z\times 2\Z\times \Z$ (viewed as a subgroup of $\H$), then $\pi^{\H}_{P\cap C}(n)\leq 72\sum_{i=0}^3 2^i{14n\choose i}$. So by Lemma \ref{lem:coset-union},  if $\VC_{\H}(P)>4(n-1)$ then we must have 
\[
n\leq \log\left(288\sum_{i=0}^3 2^i{14n\choose i}\right).
\]
One can check that the last inequality fails when $n=36$. Thus $\VC_{\H}(P)\leq 140$.
\end{proof}

\begin{remark}
    Recall that $\mathbb{H}$ is the free $2$-nilpotent group on two generators, and thus the work in this section is a first step toward a more general analysis of generalized progressions in nilpotent groups (Question \ref{Q:nil-simple}). While a description of generalized progressions analogous to Theorem \ref{thm:PNNexp} could be possible in free $2$-nilpotent groups on more generators, a similar approach in $s$-nilpotent groups for $s>2$ is likely to be much more challenging. For example, see the discussion of free nilpotent groups in \cite{TaoFNGblog}.  
\end{remark}

\section{Generalized progressions in free groups}\label{sec:free}

Throughout this section, we fix an integer $k\geq 1$ and let $F_k$ be the finitely-generated free group of rank $k$ with a fixed choice of generators $a_1,\ldots,a_k$. In the context of Question \ref{Q:allGAP}, a natural set system to consider is $\cP_{F_k}(\abar)$, i.e., the set system of all left translates of all generalized progressions in $F_k$ generated by $a_1,\ldots,a_k$ (recall Definition \ref{def:GAP}). The main result of this section will show that the VC-dimension of $\cP_{F_k}(\abar)$ is at most $3k-1$. To prove this, we will use  intuition from geometric group theory and, in particular, pseudometrics on the Cayley graph  of $F_k$ (with respect to our fixed set of generators). We refer the reader to \cite{Meier-book} for the necessary background and basic definitions. Recall that when viewed as a simple undirected graph, the Cayley graph of $F_k$ is a tree (i.e., connected and acyclic). Thus we will start this section at the basic viewpoint of graph-theoretic trees.

Let $T=(V,E)$ be a tree. We call a vertex $v\in V$ a \textbf{leaf of $T$} if $v$ has degree at most $1$ in $T$. A subset $X\seq V$ is called \textbf{connected} if the induced subgraph on $X$ is connected. Note that a subgraph $G$ of $T$ is connected if and only if $G$ is a tree. Moreover, the intersection of any family of connected subgraphs of $T$ is also connected (and thus is a tree). This allows us to make the following definition.

\begin{definition}
    Given a tree $T=(V,E)$ and a subset $X\seq V$, we  define the \textbf{minimal tree of $X$}, denoted $T(X)$, to be the intersection of all connected subgraphs of $T$ containing $X$. We  let $L(X)$ denote the set of leaves  of  $T(X)$.
\end{definition}

We will also use the following  notation for paths in trees.

\begin{definition}
Given a tree $T=(V,E)$ and $v,w\in V$, we let $\p{v}{w}\seq V$ denote the set of vertices in the unique path from $v$ to $w$ in $T$. 
\end{definition}

 Note that if $T=(V,E)$ is a tree and $X\seq V$ then $L(X)\seq X$. The next lemma shows that if $X$ can be shattered by a set system of connected sets in $T$, then in fact $L(X)=X$. This is analogous to the fact that if a set system of convex sets shatters some finite set $X$, then $X$ coincides with the set of extreme points of its convex hull (this observation is often the first step in many VC-dimension calculations for convex sets).

\begin{lemma} \label{leaf lemma}
    Let $T=(V,E)$ be a tree and fix a  subset $X\seq V$. Suppose $\cS$ is a set system on $V$ consisting of connected sets. If $\cS$ shatters $X$, then $L(X) = X$. 

    \begin{proof}
        Suppose $L(X)\subsetneq X$ and pick a vertex $x\in X\backslash L(X)$.  Let $W$ be the vertex set of $T(X)$, and let $G$ be the induced subgraph of $T(X)$ on $W\backslash\{x\}$. Since $x$ has degree at least 2 in $T(X)$, and $T(X)$ is a tree, it follows that $G$ has at least two connected components. Moreover, since $T(X)$ is the minimal tree of $X$, each connected component of $G$ contains at least one vertex in $X$. So we may choose $v,w\in X$ such that $v$ and $w$ lie in different connected components of $G$. Then $x\in \p{v}{w}\backslash\{v,w\}$ by construction. Now, if $S\seq V$ is connected and $v,w\in S$, then $\p{v}{w}\seq S$, hence $x\in S$. So $\cS$ does not cut out $\{v,w\}$ from $X$, hence $\cS$ does not shatter $X$.
    \end{proof}
\end{lemma}

\begin{remark}\label{rem:closestpoint}
If $T=(V,E)$ is a tree then for any $v\in V$ and any nonempty connected set $X\seq V$, there is a unique point in $X$ with minimum path distance to $v$ (among points in $X$).
\end{remark}


We now return to the free group $F_k$ with fixed generators $a_1,\ldots,a_k$.  In order to further formalize the geometric intuition for generalized progressions, we will define a family of pseudometrics for which generalized progressions are similar to metric balls (Remark \ref{rem:F_k progression-convexity}). These pseudometrics also sum to the standard word metric on $F_k$, which is of course a well-studied object (see, e.g., \cite[Section 9.1]{Meier-book}).

\begin{definition}$~$
\begin{enumerate} 
    \item For each $i\in[k]$, define $d_i\colon F_k\times F_k\to \N$ so that $d_i(x,y)$ is the total number of  occurrences of $a_i$ and $a_i\inv$ in the unique reduced word for $x\inv y$ (with respect to our fixed set of generators).
    \item Define $d\colon F_k\times F_k\to \N$ so that $d(x,y)=\sum_{i=1}^k d_i(x,y)$.
\end{enumerate}
\end{definition}

\begin{remark}$~$ \label{rem:pseudometric_props}
\begin{enumerate}[$(a)$]
    \item It is clear that each $d_i$ is left invariant (and thus so is $d$).
    \item Each $d_i$ is a pseudometric (and thus so is $d$). For completeness, we verify the triangle inequality. Fix $i\in[k]$. Given $x \in F_k$, let $n_{i}(x)$ denote the total number of  occurrences of $a_i$ and $a_i\inv$ in the reduced word for $x$. Note that for any $x,y\in F_k$, we have $n_i(xy)\leq n_i(x)+n_i(y)$ (since word reduction can only reduce the number of $a_i$ and $a_i\inv$). So given  $x,y,z\in F_k$, 
    \[
    d_i(x,y)=n_i(x\inv y)=n_i(x\inv z z\inv y)\leq n_i(x\inv z)+n_i(z\inv y)=d_i(x,z)+d_i(z,y).
    \]
    \item The pseudometric $d$ is in fact a metric. Indeed, $d$ is precisely the standard word metric on $F_k$. Recall also that $d$ coincides with the path metric in the Cayley graph on $F_k$ (with respect to our fixed generators).
\end{enumerate}
\end{remark}

For the rest of the section we identify $F_k$ with its  Cayley graph (as a simple undirected graph). When there is no possibility for confusion, we will also identify a subset $X\seq F_k$ with the subgraph on $X$ induced from  $F_k$. Recall again that (the Cayley graph of) $F_k$ is a tree.

\begin{proposition} \label{triangle-equality}
    Given $x,y \in F_k$, if $z\in \p{x}{y}$ then for all $i \in [k]$, $d_i(x,y)=d_i(x,z) + d_i(z,y)$. 
\end{proposition}
\begin{proof}
The argument is similar to the analogous statment for the word metric $d$ (which is standard). In particular, the key point is that $d_i(x,y)$ counts the number of occurrences of $a_i$ and $a_i\inv$ as labels on the path  from $x$ to $y$ in the Cayley graph. Details are left to the reader.
\end{proof}

We now turn our attention  to generalized progressions in $F_k$ generated by $a_1,\ldots,a_n$. Since our generators are fixed, we will simplify the notation from Definition \ref{def:GAP} as follows.

\begin{definition}
    Given $\bar{N}\in \N^k$, let $P(\bar{N})=P(\abar,\bar{N})$. Let $\cP_k$ denote the set system $\cP_{F_k}(\abar)$. 
\end{definition}

Recall that $\cP_k$ consists of all left translates of all generalized progressions in $F_k$ generated by $\abar$, i.e., $\cP_k=\{gP(\bar{N}):\bar{N}\in\N^k,~g\in F_k\}$.

Our first goal is to show that if a sufficiently large set $X\seq F_k$ is shattered by $\cP_k$, then the tree $T(X)$ admits a partition with certain special properties (see Proposition \ref{prop:structure_of_3k_weak} for the precise statement). This will be done in a few steps.

\begin{remark} \label{rem:F_k progression-convexity}
    Fix $\bar{N} = (N_1, \dots, N_k) \in \N^k$ and $g\in F_k$. Then 
\[
gP(\bar{N})=\{x\in F_k:d_i(g,x)\leq N_i\text{ for all }i\in [k]\}.
\]
It follows from Proposition \ref{triangle-equality} that $gP(\bar{N})$ is connected. 
\end{remark}

\begin{lemma}\label{lem:fork}
Fix $gP(\bar{N})\in \cP_k$ and $x\in F_k$. Assume that for all $i\in[k]$, there are $y\in gP(\bar{N})$ and $p\in \p{g}{y}$ such that $d_i(p,x)\leq d_i(p,y)$. Then $x\in gP(\bar{N})$. 
\end{lemma}
\begin{proof}
    We fix $i\in[k]$ and show $d_i(g,x)\leq N_i$ (recall Remark \ref{rem:F_k progression-convexity}). By assumption there is $y\in gP(\bar{N})$ such that $p\in \p{g}{y}$ and $d_i(p,x)\leq d_i(p,y)$. Note $d_i(g,y)\leq N_i$. By the triangle inequality and Proposition \ref{triangle-equality}, we have $d_i(g,x)\leq d_i(g,p)+d_i(p,x)\leq d_i(g,p)+d_i(p,y)=d_i(g,y)\leq N_i$.
\end{proof}

\begin{lemma} \label{entry point}
    Let $X \seq F_k$ be connected, and assume $X\cap gP(\bar{N})\neq\emptyset$ for some $g\in F_k$ and $\bar{N}\in\N^k$. Then there are $h\in X$ and $\bar{M}\in\N^k$ such that $X \cap gP(\bar{N}) = X \cap hP(\bar{M})$. 
\end{lemma}
\begin{proof}
     Choose $h$ to be the point in $X$ with minimum $d$-distance to $g$ (recall Remark \ref{rem:closestpoint}). Then for any $p\in X$, we have $h\in t(g,p)$ and so by Proposition \ref{triangle-equality}, $d_i(g,p)=d_i(g,h)+d_i(h,p)$ for all $i\in [k]$. Since $X\cap gP(\bar{N})\neq\emptyset$, this implies $d_i(g,h)\leq N_i$ for all $i\in [k]$ (recall Remark \ref{rem:F_k progression-convexity}). Moreover, given any $p\in X$ and $i\in[k]$, we have $d_i(g,p)\leq N_i$ if and only if $d_i(h,p)\leq N_i-d_i(g,h)$. Therefore, if we set $M_i=N_i-d_i(g,h)$, then $\bar{M}\in\N^k$ and $X \cap gP(\bar{N}) = X \cap hP(\bar{M})$.  
\end{proof}

\begin{definition}\label{def:branch_count_free_group}
    Given $X\seq F_k$ and $p\in T(X)$, define $\text{br}_X(p)$ to be the set of connected components of $T(X)\backslash\{p\}$. We also define $\text{br}_X^*(p) := \text{br}_X(p) \cup \{\{p\}\}$.
\end{definition}

\begin{remark} \label{branch-triangle-inequality}
     In the context of the previous definition, note that the sets in $\br_X^*(p)$ form a partition of $T(X)$. Moreover, if two points $x,y\in T(X)$ lie in different pieces of this partition, then $p\in \p{x}{y}$. 
\end{remark}

\begin{lemma}\label{lem:double_pitch_fork_bad}
    Let $X\seq F_k$ be finite, and suppose that there is a point $p\in T(X)$ admitting a partition $\br_X(p)=\cB_0\sqcup \cB_1$ with $|X\cap\bigcup \cB_0|,|X\cap \bigcup \cB_1|\geq k+1$. Then $\cP_k$ does not shatter $X$. 
\end{lemma}
\begin{proof}
For $t\in\{0,1\}$, set $X_t=X\cap \bigcup \cB_t$. Given $i\in [k]$ and $t\in\{0,1\}$, choose $x_{t,i}\in X_t$ so that $d_i(p,x_{t,i})=\max\{d_i(p,x):x\in X_t\}$ (this is possible since $X$ is finite). Define 
\[
A=\{x_{t,i}:i\in[k],~t\in\{0,1\}\},
\] 
and note that $A\seq X$ by construction. We claim that $\cP_k$ does not cut out $A$ from $X$. Suppose otherwise. Then by Lemma \ref{entry point} (applied to the connected set $T(X)$), there are $g\in T(X)$ and $\bar{N}\in\N^k$ such that $X\cap gP(\bar{N})=A$.  Without loss of generality, assume $g\in \{p\}\cup\bigcup\cB_0$. Note that $A\cap X_1=\{x_{1,i}:i\in [k]\}$, and hence $|A\cap X_1|\leq k$. Since $|X_1|\geq k+1$, we may fix some element $y\in X_1\backslash A$. Given $i\in [k]$, we have $x_{1,i}\in A\seq gP(\bar{N})$, $p\in \p{g}{x_{1,i}}$ (by Remark \ref{branch-triangle-inequality}, and since $g$ and $x_{1,i}$ lie in different sets in $\br^*_X(p)$ by construction), and $d_i(p,y)\leq d_i(y,x_{1,i})$ (by choice of $x_{1,i}$, and since $y\in X_1$). So $y\in gP(\bar{N})$ by Lemma \ref{lem:fork}, which contradicts $y\in X\backslash A$ and $X\cap gP(\bar{N})=A$.
\end{proof}

\begin{remark}\label{rem:3.19}
The proof of Lemma \ref{lem:double_pitch_fork_bad} is reminiscent of the standard argument  that $2k$ is an upper bound for the VC-dimension of axis-parallel boxes in $\R^k$ (Example \ref{ex:box-R}), in which  one picks out the points of ``extremal coordinates" from a set of size $2k + 1$.  Lemma \ref{lem:double_pitch_fork_bad} is guided by this idea, complicated by the fact that we are working over a family of pseudometrics. This will also provide  intuition for Proposition \ref{prop:structure_of_3k_weak} and Definition \ref{dominating-set} below.
\end{remark}

\begin{proposition}\label{prop:structure_of_3k_weak}
    Fix $X \seq F_k$ with $|X|=3k$. If $\cP_k$ shatters $X$, then there is a point $p\in T(X) \backslash X$ such that $|\br_X(p)|=3$ and $|X\cap B|=k$ for all $B\in\br_X(p)$.
\end{proposition}
\begin{proof}
    Suppose $\cP_k$ shatters $X$. By Lemma \ref{leaf lemma} (and Remark \ref{rem:F_k progression-convexity}), we have $L(X)=X$. Given a collection $\cB$ of subsets of $F_k$, let $\epsilon(\cB)\coloneqq |X\cap \bigcup \cB|$. Given $p\in T(X)$, we define
    \[
    N(p)=\min_{\substack{\cB\seq \rm{br}(p)\\
    \epsilon(\cB)\geq k+1}}\epsilon(\cB)
    \]
    (this is well-defined since $\epsilon(\br_X(p))=|X\backslash \{p\}|\geq 3k-1\geq k+1$).
    
        Let $N = \min_{p \in T(X)} N(p)$ and define $P = \{ p \in T(X) : N(p) = N\}$. We can then choose $p\in P$ and $\cB_0\seq \br_X(p)$ such that $\epsilon(\cB_0)=N$ and $\cB_0$ is ``maximal" in the sense that for any $q\in P$ and $\cB\seq\br_X(q)$, if $\epsilon(\cB)=N$ then $|\cB|\leq|\cB_0|$. \medskip
        
        \noindent\emph{Claim 1.} $|\cB_0|>1$.
        
        \noindent\emph{Proof.} Suppose $\cB_0$ is a singleton $\{B\}$. Recall that $B$ is a connected subset of $T(X)$. We note that $B$ must contain a vertex of degree at least $3$ in $T(X)$, since otherwise $|X\cap B|=1$ (recall $L(X)=X$), which contradicts $\epsilon(\cB_0)\geq k+1$. Let $q\in B$ be the vertex of degree at least $3$ in $T(X)$ closest to $p$ (with respect to the path metric $d$).  Set $\cB=\{B\in\br_X(q):p\not\in B\}$. Then it follows from the choice of $q$ that $|\cB|\geq 2$ and $X\cap B=X\cap \bigcup\cB$. Hence $\epsilon(\cB)=\epsilon(\cB_0)=N\geq k+1$. Therefore $N(q)\geq N$, and thus $N(q)=N$ by minimality of $N$. Altogether, we have $q\in P$ and $\cB\seq \br_X(q)$ with $\epsilon(\cB)=N$ and $|\cB|\geq 2>|\cB_0|$, which contradicts the maximality of $\cB_0$.\clqed \medskip
        
        \noindent\emph{Claim 2.} $p\not\in X$.
        
        \noindent\emph{Proof.} This is immediate from Claim 1 since if $p\in X$ then $|\br_X(p)|=1$.\clqed\medskip
        
        \noindent\emph{Claim 3.} If $B\in\cB_0$ then $|X\cap B|\leq k$.
        
        \noindent\emph{Proof.} Fix $B\in\cB_0$ and set $\cB=\{B\}$ (so $\epsilon(\cB)=|X\cap B|$). Since $L(X)=X$, it follows from Claim 1 that $\epsilon(\cB)<\epsilon(\cB_0)=N(p)$. So by definition of $N(p)$, we must have $\epsilon(\cB)\leq k$.\clqed \medskip
        
        \noindent\emph{Claim 4.} $N=2k$.
        
        \noindent\emph{Proof.} First, suppose $N\geq 2k+1$. Fix $B\in\cB_0$ and let $\cB=\cB_0\backslash\{B\}$. By Claim 3, $\epsilon(\cB)\geq \epsilon(\cB_0)-k=N-k\geq k+1$. On the other hand, as in the proof of Claim 1, we have $\epsilon(\cB)<\epsilon(\cB_0)=N$ (since $L(X)=X$). Altogether $k+1\leq \epsilon(\cB)<N=N(p)$, which contradicts the definition of $N(p)$. So  $N\leq 2k$.
        
        Now suppose $N\leq 2k-1$. Let $\cB_1=\br_X(p)\backslash \cB_0$, and consider the partition $\br_X(p)=\cB_0\sqcup\cB_1$. We have $\epsilon(\cB_0)=N\geq k+1$. On the other hand, since  $L(X)=X$ and $p\not\in X$ (by Claim 2), we also have $\epsilon(\cB_0)+\epsilon(\cB_1)=\epsilon(\br_X(p))=|X|=3k$, and thus  $\epsilon(\cB_1)=3k-\epsilon(\cB_0)=3k-N\geq k+1$. But  since we have assumed $\cP_k$ shatters $X$, this contradicts Lemma \ref{lem:double_pitch_fork_bad}. Therefore $N=2k$.\clqed\medskip
        
        Now set $\cB=\br_X(p)\backslash\cB_0$, and fix a nontrivial partition $\cB_0=\cC\sqcup\cD$ (which is possible by Claim 1). We  show $\epsilon(\cB)=\epsilon(\cC)=\epsilon(\cD)=k$.  Since $L(X)=X$ and $p\not\in X$, we have $\epsilon(\cB)+\epsilon(\cC)+\epsilon(\cD)=|X|=3k$. Also, $\epsilon(\cC)+\epsilon(\cD)=\epsilon(\cB_0)=N=2k$ (by Claim 4). So $\epsilon(\cB)=k$. This implies $\epsilon(\cB\cup \cC)\geq k+1$, and thus $\epsilon(\cB)+\epsilon(\cC)=\epsilon(\cB\cup \cC)\geq N$ (recall $N=N(p)$). So $\epsilon(\cC)\geq N-\epsilon(\cB)=k$. Similarly, $\epsilon(\cD)\geq k$. Since $\epsilon(\cC)+\epsilon(\cD)=2k$, this yields $\epsilon(\cC)=\epsilon(\cD)=k$.
        
        To finish the proof, we just need to show that $\cB$, $\cC$, and $\cD$ are singletons. For a contradiction, suppose $|\cB|>1$, and fix some $B\in\cB$. Then $1\leq |X\cap B|<\epsilon(\cB)=k$. Setting  $\cC'=\cC\cup\{B\}$, we have $\epsilon(\cC')=\epsilon(\cC)+|X\cap B|$, and thus $k+1\leq \epsilon(\cC')<2k$, which contradicts $N(p)=N=2k$. By a similar argument, $\cC$ and $\cD$ are singletons.
  \end{proof}

To establish an upper bound on the VC-dimension of $\cP_k$, we will show that a partition of $T(X)$ as in the previous proposition leads to a subset of $X$ that cannot be cut out by $\cP_k$ (which then  shows that $\VC(\cP_k)\geq 3k$ is impossible). This will be done using the following technical notion.

\begin{definition}\label{dominating-set}
    Fix a finite set $X \seq F_k$  and a point $p \in T(X)$. A \textbf{dominating sequence for $X$ with respect to $p$} is a sequence $(f_B)_{B \in \text{br}^*_X(p) } $ satisfying the following properties for each $B\in\br^*_X(p)$:
    \begin{enumerate}[$(i)$]
        \item  $f_B$ is a function from $[k]$ to $X \backslash B$. 
        \item For all $i\in [k]$, if $x\in X\backslash B$ then $d_i(f_B(i),p)\geq d_i(x,p)$.
        \item For all $i\in [k]$, if $f_{\{p\}}(i)\in X\backslash B$ then $f_B(i)=f_{\{p\}}(i)$.
    \end{enumerate}
    We  call $\bigcup_{B \in \text{br}^*_X(p)} \text{im}(f_B)$  the \textbf{image} of  $(f_B)_{B \in \text{br}^*_X(p) } $.
\end{definition}

\begin{remark}\label{rem:DSexist}
Given arbitrary $X\seq F_k$ and $p\in T(X)$, a dominating sequence for $X$ with respect to $p$ need not exist (e.g., if $X=\{p\}$). However, if $p\not\in X$ and $X\backslash B\neq\emptyset$ for all $B\in\br^*_X(p)$, then one can directly construct a dominating sequence as follows:
\begin{enumerate}[$(1)$]
\item First define $f_{\{p\}}\colon [k]\to X$ so that for any $i \in [k]$, $d_i(f_{\{p\}}(i),p) = \max\{d_i(x,p) :x \in X\}$.
\item Then, given $B\in\br_X(p)$, define $f_B\colon [k]\to X\backslash B$ so that for any $i\in [k]$, if $f_{\{p\}}(i)\in X\backslash B$ then $f_B(i)=f_{\{p\}}(i)$, and otherwise $d_i(f_B(i),p) = \max\{d_i(x,p) : x \in X\backslash B\}$.
\end{enumerate}
       A special case of the above situation is when $p\in T(X)\backslash X$.  As indicated by Proposition \ref{prop:structure_of_3k_weak}, this will ultimately be the setting we care about.
\end{remark}

Next we establish some basic properties of dominating sequences. 

\begin{proposition}\label{prop:DSprops}
Fix a finite set $X\seq F_k$ and a point $p\in T(X)$. Assume $(f_B)_{B\in\br^*_X(p)}$ is a dominating sequence for $X$ with respect to $p$, with image $A$.
\begin{enumerate}[$(a)$]
\item Suppose $B\in\br^*_X(p)$ and $\im(f_B)\seq gP(\bar{N})$ for some $g\in B$ and $\bar{N}\in\N^k$. Then $X\backslash B\seq gP(\bar{N})$.  

\item Suppose $\cP_k$ cuts out $A$ from $X$. Then there is some $B\in\br^*_X(p)$ such that $X\backslash B\seq A$. 

\item $|A|\leq 2k$.
\end{enumerate}
\end{proposition}
\begin{proof}
    
  Part $(a)$.  Fix $x \in X \backslash B$. Given $i\in[k]$, we have $f_B(i)\in \im(f_B)\seq gP(\bar{N})$, $p\in \p{g}{f_B(i)}$ (by Remark \ref{branch-triangle-inequality}, and since $g\in B$ and $f_B(i)\in X\backslash B$), and $d_i(p,x)\leq d_i(p,f_B(i))$ (by definition of a dominating sequence, and since $x\in X\backslash B$). So $x\in gP(\bar{N})$ by Lemma \ref{lem:fork}.

   Part $(b)$. By assumption and Lemma \ref{entry point} (applied to the connected set $T(X)$) there is some $gP(\bar{N})\in \cP_k$ such that $X\cap g P(\bar{N})=A$ and $g\in T(X)$. So $g\in B$ for some $B\in\br^*_X(p)$. Since $\im(f_B)\seq A\seq gP(\bar{N})$, we have $X\backslash B\seq gP(\bar{N})$ by part $(a)$. So $X\backslash B\seq X\cap gP(\bar{N})=A$.

Part $(c)$. First note that if $i\in [k]$, $B\in\br_X(p)$, and $f_B(i) \neq f_{\{p\}}(i)$, then  $f_{\{p\}}(i) \in B$ (by definition of a dominating sequence). Hence, for any $B\in \br_X(p)$, we have $\im(f_B) \backslash \im(f_{\{p\}})\seq f_B(f\inv_{\{p\}}(B))$, and so $|\im(f_B) \backslash \im(f_{\{p\}})|\leq |f_{\{p\}}\inv(B)|$. Note also that if $B,B'\in \br_X(p)$ are distinct, then $f\inv_{\{p\}}(B)$ and $f\inv_{\{p\}}(B')$ are disjoint subsets of $[k]$. Altogether,
    \[
    |A| \leq |\im(f_{\{p\}})| + \sum_{B \in \br(p)} |\im(f_B) \backslash \im(f_{\{p\}})| \leq k + \sum_{B \in \br(p)} |f_{\{p\}}\inv(B)|\leq 2k.\qedhere
    \] 
\end{proof}

We can now prove the main result of this section.

\begin{theorem}\label{thm:free3k-1}
     $\VC_{F_k}(\cP_k)\leq 3k-1$.
\end{theorem}
\begin{proof} 
Fix $X\seq F_k$ with $|X|=3k$ and assume, towards a contradiction, that $\cP_k$ shatters $X$. By Proposition \ref{prop:structure_of_3k_weak}, there is a point $p\in T(X)\backslash X$ such that $\br_X(p)=\{B,C,D\}$ with $|X\cap B|=|X\cap C|=|X\cap D|=k$. 

Note that $\br^*_X(p)=\{\{p\},B,C,D\}$. Fix a dominating sequence $(f_{\{p\}},f_B,f_C,f_D)$  for $X$ with respect to $p$  (recall Remark \ref{rem:DSexist}). Let $A$ be the image of this sequence.\medskip

\noindent\emph{Claim.} There is some $E\in \{B,C,D\}$ such that $X\backslash E=A$ and, for all $Z\in\br^*_X(p)\backslash \{E\}$, $X\backslash Z\not\seq A$. 

\noindent\emph{Proof.} By assumption, $\cP_k$ cuts out $A$ from $X$ and thus, by Proposition \ref{prop:DSprops}$(b)$, there is some $E\in\br^*_X(p)$ such that $X\backslash E\seq  A$. Note that $X\backslash \{p\}=X\not\seq A$ since $|X|=3k$ and $|A|\leq 2k$ (by Proposition \ref{prop:DSprops}$(c)$). So $E\in\{B,C,D\}$. Since $|X\cap E|=k$, we then have $|X\backslash E|\geq |X|-k=2k$, and hence $X\backslash E=A$. Finally, if $Z\in \br^*_X(p)\backslash \{E\}$, then $X\backslash Z\not\seq A$ since otherwise (using $A=X\backslash E$) we would have $E\seq Z$.\clqed\medskip

Let $E$ be as in the claim. Since $\cP_k$ cuts out $A$ from $X$, we may apply Lemma \ref{entry point} (to the connected set $T(X)$) to find $g\in T(X)$ and $\bar{N}\in\N^k$  such that $X\cap gP(\bar{N})=A$. So $g$ is an element of some set in $\br^*_X(p)$. Given $Z\in\br^*_X(p)$, we have $\im(f_Z)\seq A\seq gP(\bar{N})$ and thus, if $g\in Z$ then $X\backslash Z\seq gP(\bar{N})$ by Proposition \ref{prop:DSprops}$(a)$. Since $X\cap gP(\bar{N})=A$, it then follows from the claim that $g\in E$. 

Finally, we show $X\seq gP(\bar{N})$, which  contradicts $X\not\seq A$. Fix $x\in X$. Given $i\in [k]$, we have $f_{\{p\}}(i)\in A\seq gP(\bar{N})$, $p\in\p{g}{f_{\{p\}}(i)}$ (by Remark \ref{branch-triangle-inequality}, and since $f_{\{p\}}(i)\in A=X\backslash E$ and $g\in E$), and $d_i(x,p)\leq d_i(f_{\{p\}}(i),p)$ (by definition of a dominating sequence, and since $p\not\in X$). So $x\in gP(\bar{N})$ by Lemma \ref{lem:fork}.
    \end{proof}

   The ingredients of the proof of Theorem \ref{thm:free3k-1} suggest that a better bound on $\VC(\cP_k)$ is possible. For starters, the image of a dominating sequence has size at most $2k$, and thus it is reasonable to suspect that the upper bound can be improved to $2k$. Indeed the nature of the proof suggests that such a bound is possible with a suitable ``$(2k+1)$-analog" of Proposition \ref{prop:structure_of_3k_weak}. We suspect that any inefficiency in the proof of Proposition \ref{prop:structure_of_3k_weak} resides in the fact that we are using Lemma \ref{lem:double_pitch_fork_bad}: the hypotheses used are a rather special case and as we noted in Remark \ref{rem:3.19}, Proposition \ref{prop:structure_of_3k_weak} was conceptualized by chasing a situation where Lemma \ref{lem:double_pitch_fork_bad} is satisfied. 

   Of course, the other obvious question is about a lower bound for $\VC(\cP_k)$. Here again, it is reasonable to tentatively guess that $\VC(\cP_k)\geq 2k$ based on the intuition from Example \ref{ex:box-R}. While explicit lower bound constructions are problems of a different nature, we can at least establish the following basic lower bound.

\begin{proposition} \label{prop:freelower}
Given $\bar{N}\in\N^k$, if $N_i\geq 1$ for all $i\in[k]$ then $\VC_{F_k}(P(\bar{N}))\geq k$.
\end{proposition}

\begin{proof}
    Fix $\bar{N}\in\N^k$ with $N_i\geq 1$ for all $i\in [k]$. We claim that the set system of left translates of $P(\bar{N})$ shatters the set of generators $X = \{a_1, \dotsc, a_k\}$. First, note that $X\cap a_1^{N_1+2}P(\bar{N})=\emptyset$. So now fix a nonempty subset $Y\seq X$. Fix  $j\in[k]$ such that $a_j\in Y$ and enumerate $X\backslash Y=\{a_{i_1},\ldots,a_{i_t}\}$ for some $t<k$. Define $g=a_ja_{i_1}^{N_{i_1}}\ldots a_{i_t}^{N_{i_t}}$. We claim that $X\cap gP(\bar{N})=Y$. Given $i\in [k]$, we have
       \begin{align*}
    	g\inv a_i = 	\begin{cases}
    		a_{i_t}^{\nv N_{i_t}}\ldots a_{i_1}^{\nv N_{i_1}} & \text{if $i=j$,} \\
    		a_{i_t}^{\nv N_{i_t}}\ldots a^{\nv N_{i_1}}_{i_1}a_j\inv a_i & \text{otherwise.}
    	\end{cases}
    \end{align*} 
    Since $N_i\geq 1$ for all $i\in [k]$, it follows that $g\inv a_i\in P(\bar{N})$ if and only if $a_i\in Y$, i.e., $X\cap gP(\bar{N})=Y$.      
\end{proof}

\begin{remark}\label{rem:lower}
    Recall that $F_1$ is the additive group of integers $\Z$ and $\cP_1$ is the collection of finite intervals in $\Z$. It is easy to check that $\VC(\cP_1)=2$ (e.g., $\cP_1$ shatters any two distinct points in $\Z$). Thus we see already that the lower bound in Proposition \ref{prop:freelower} is too small when $k=1$. We will also show in the next example that $\VC(\cP_2)\geq 4$. 
\end{remark}

\begin{example}\label{ex:4-set in F_2}
$\cP_2$ shatters $\{a_1^{10} ,a_2^{\nv 10}, a_1^{\nv 5}, a_2^5a_1^{3}\}$ in $F_2$. 

  To see this, first set $w \coloneqq a_1^{10}$, $x \coloneqq a_2^{\nv 10}$, $y \coloneqq a_1^{\nv 5}$, and $z \coloneqq a_2^5a_1^{3}$. Let $X=\{w,x,y,z\}$. The following tables provide elements in $\cP_2$ cutting out all subsets of $X$.

\begin{center}
\begin{tblr}{
    colspec      = {*{8}{c}},
    vline{1-6}     = {0.5pt},
    vline{7-9}   = {1-5}{0.5pt},
    hline{1,2}     = {1,2,4,5,7,8}{0.5pt},
    hline{6}   = {7,8}{0.5pt},
    hline{8}  = {1,2,4,5}{0.5pt}
}
     $\seq X$ & $\in \cP_2$ & & $\seq X$ & $\in \cP_2$ & & $\seq X$ & $\in \cP_2$  \\ 
    $\emptyset$ & $P(1, 1)$ & &  $\{w, x\}$ & $wP(10, 10)$ & &  $\{w,x,y\}$ & $xP(10, 10)$ \\ 
     $X$ & $P(10, 10)$ & &  $\{w, y\}$ & $wP(10, 1)$ & &  $\{w,x,z\}$ & $wP(15, 10)$ \\ 
      $\{w\}$ & $wP(1, 1)$ & &  $\{w, z\}$ & $wP(13, 5)$ & &  $\{w,y,z\}$ & $wP(15, 5)$ \\ 
     $\{x\}$ & $xP(1, 1)$  & &  $\{x, y\}$ & $xP(10, 5)$ & &  $\{x,y,z\}$ & $zP(8, 15)$ \\
     $\{y\}$ & $yP(1, 1)$ & &  $\{x, z\}$ & $zP(3, 15)$ & & & \\
     $\{z\}$ & $zP(1, 1)$ & &  $\{y, z\}$ & $zP(8, 5)$ & & & 
\end{tblr} 
\end{center}  

Specifically, in each of the three tables, we claim that if $S$ is a subset of $X$ from the left hand column, and $P$ is the corresponding element of $\cP_2$ in the  right hand column, then $S=X\cap P$. This can be verified using Remark \ref{rem:F_k progression-convexity}, which reduces the question to computing $d_1$ and $d_2$ distances between various elements of $X\cup\{e\}$. These distances are provided in the table below (each entry gives $(d_1(u,v),d_2(u,v))$ for the corresponding $u,v\in X\cup\{e\}$).
\begin{center}
\begin{tabular}{c|cccc}
  $x$ & $(10,10)$ & & & \\
  $y$ & $(15,10)$ & $(5,10)$ & & \\
  $z$ & $(13,5)$ & $(3,15)$ & $(8,5)$ & \\
  $e$ & $(10,0)$ & $(0,10)$ & $(5,0)$ & $(3,5)$ \\ \hline
 & $w$ & $x$ & $y$ & $z$ 
\end{tabular}
\end{center}
\end{example}

\begin{remark}
    Note $\cP_1$ is the collection of all finite connected sets in (the Cayley graph of) $F_1=\Z$, and can be seen as the discrete analogue of (finite length) convex sets in $\R$. Given the thematic connection between geometric convexity and graph-theoretic connectedness employed in this section, it is natural to ask what happens for $k>1$. In $\R^2$, the set system of convex sets has infinite VC-dimension since, for example, one can clearly shatter any  set of points lying on a circle. We can use a similar argument to see that if $k>1$ then the set system  of all connected sets in $F_k$ has infinite VC-dimension. More generally, let $T=(V,E)$ be a tree and let $\cS$ be the set system on $V$ consisting of all connected sets. Then  $\cS$ shatters any set  $X\seq V$ such that $X=L(X)$ since   if $Y\seq X$ then $T(Y)\in \cS$ and $T(Y)\cap X=Y$. Returning to $F_k$, if $k>1$ then there are  arbitrarily large $X\seq F_k$ such that $L(X)=X$ (e.g.,  $X_n=\{x\in F_k:d(e,x)=n\}$ for  $n\geq 1$).
\end{remark}

\begin{remark}
Given the universality of free groups among finitely generated groups, one might hope that Theorem \ref{thm:free3k-1} can be used to obtain a positive answer to Question \ref{Q:allGAP}. In particular, let  $G$ be a group generated by $a'_1,\ldots,a'_k$, and let $\varphi\colon F_k\to G$ be a surjective homomorphism sending $a_i$ to $a'_i$. Set $H=\ker\varphi$, and let $\cP_kH$ denote the set system on $F_k$ consisting of all sets of the form $PH$ for $P\in\cP_k$. Then one can show $\VC(\cP_G(\abar))=\VC(\cP_kH)$. However, there is no clear way to bound $\VC(\cP_kH)$ using a bound on $\VC(\cP_k)$. In fact, in full generality, this is not possible. For example, let $F$ be the additive group of a pseudofinite field of characteristic not $2$, and set $S=\{x^2:x\in F\}$. Then $\VC_F(S)=\infty$ by a result of Duret \cite{Duret}. Now let $G=F\times F$ and $A=\{(x,y)\in G:y=x^2\}$. Then $\VC_G(A)<\infty$ since $A$ is defined by a polynomial equation and, in any field, all quantifier-free (bipartitioned) formulas are NIP. On the other hand, if $H=F\times \{0\}$ then  $A+H=F\times S$, and so $\VC_G(A+H)=\VC_G(F\times S)=\VC_F(S)=\infty$. 
\end{remark}

\begin{remark}
 From the perspective of geometric group theory, another natural set system on $F_k$ to consider is the collection of all balls in the word metric. However, an elementary argument shows that the VC-dimension of this set system is $2$. More generally, consider an arbitrary tree $T=(V,E)$ and let $\cS$ be the set system of all balls in $V$ with respect to the path metric $d$. We claim that $\VC(\cS)\leq 2$. To see this, fix distinct points $a,b,c\in V$. Let $x$ be the point in $\p{a}{b}$ with minimum distance to $c$.  Then we have $T(\{a,b,c\})=\p{a}{x}\cup\p{b}{x}\cup\p{c}{x}$, and $x$ is also the  point in $\p{b}{c}$ with minimum distance to $a$, as well as the  point in $\p{a}{c}$ with minimum distance to $b$. So after permuting $a,b,c$, we may assume $d(x,a)\leq \min\{d(x,b),d(x,c)\}$. Fix a ball $B\in \cS$ with center $y$, and suppose $b,c\in B$. We claim $a\in B$, which shows  $\cS$ does not shatter $\{a,b,c\}$. Let $z$ be the point in $T(\{a,b,c\})$ with minimum distance to $y$. Since the roles of $b$ and $c$ can be exchanged, we may assume $z\in \p{a}{x}\cup\p{b}{x}$. If $z\in\p{a}{x}$ then
 \[
d(y,a)= d(y,z)+d(z,a)\leq d(y,x)+d(x,a)\leq d(y,x)+d(x,b)=d(y,b),
\]
and hence $a\in B$ since $b\in B$. On the other hand, if $z\in\p{b}{x}$ then 
\[
d(y,a)=d(y,z)+d(z,x)+d(x,a)\leq d(y,z)+d(z,x)+d(x,c)=d(y,c),
\]
and so $a\in B$ since $c\in B$.   
\end{remark}

\begin{remark}\label{rem:free-group-NIP}
To continue with the discussion at the end of the introduction, we note that the set system $\cP_k$ is uniformly quantifier-free definable (with parameters) in the structure $(F_k,\cdot,{}\inv,\leq_1,\ldots,\leq_k)$ where $\leq_i$ is the partial order on $F_k$ given by $d_i(e,x)\leq d_i(e,y)$. For example, if we define the formula
\[
\varphi(x;y,z_1,\ldots,z_k)\coloneqq \bigwedge_{i\leq k}(y\inv x\leq_i z_i\wedge z_ia_1=a_1z_i),
\]
then any nonempty set defined by $\varphi(x;g,h_1,\ldots,h_n)$ for some $g,h_1,\ldots,h_k\in F_k$ is of the form $gP(\bar{N})$ for some $\bar{N}\in\N^k$ (this follows from Remark \ref{rem:F_k progression-convexity} and the fact that the centralizer of $a_1$ is the subgroup generated by $a_1$). 
It would be interesting to know whether this structure is NIP, or at least whether all quantifier-free bi-partitioned formulas are NIP. More specifically, to deduce finiteness of $\VC(\cP_k)$ through this route, one would only need to establish that the formula $\psi_i(x;y,z)\coloneqq  y\inv x\leq_i z$ is NIP for any $i\leq k$.
\end{remark}

\end{document}